\documentclass[10.5pt, a4paper]{amsart}
\usepackage{amssymb, amsmath, amscd, bm, amsthm, mathtools, pdfpages, setspace, subcaption, graphicx, xfrac, empheq, enumitem, array, arydshln, hyperref, relsize}

\usepackage[
backend=biber,
style=numeric-comp,
sortcites=true,
maxnames=99
]{biblatex}

\DeclareFieldFormat{pagetotal}{#1 pp}
\renewbibmacro*{note+pages}{
	\printfield{note}
	\setunit{\bibpagespunct}
	\printfield{pages}
	\setunit{\addcomma\space}
	\printfield{pagetotal}
	\clearfield{pagetotal}
}
\renewbibmacro*{chapter+pages}{
	\printfield{chapter}
	\setunit{\bibpagespunct}
	\printfield{pages}
	\setunit{\addcomma\space}
	\printfield{pagetotal}
	\clearfield{pagetotal}
}
\def\@Rref#1{\hbox{\rm \ref{#1}}}
\def\Rref#1{\@Rref{#1}}
\theoremstyle{plain}
\newtheorem{theorem}{Theorem}[section]
\newtheorem{proposition}[theorem]{Proposition}

\newtheorem{assumption}[theorem]{Assumption}
\newtheorem{lemma}[theorem]{Lemma}

\theoremstyle{definition}
\newtheorem{definition}{Definition}[section]

\newtheorem{remark}[definition]{Remark}

\newcommand{\Ls}{\mathrm{L}^{2}}
\newcommand{\Hs}{\mathrm{H}^{1}_{0}}

\newcommand{\Hso}{\mathrm{H}^{1}}
\newcommand{\Ht}{\mathcal{H}_{\tau}}
\newcommand{\sd}{\ensuremath{\mathrm{d}}}
\DeclareMathOperator{\trace}{trace}
\DeclareMathOperator{\blkdiag}{blkdiag}
\DeclareMathOperator{\diag}{diag}
\DeclareMathOperator{\Ran}{Ran}
\DeclareMathOperator{\Ker}{Ker}

\DeclareFontFamily{U}{mathx}{\hyphenchar\font45}
\DeclareFontShape{U}{mathx}{m}{n}{
	<5><6><7><8><9><10>
	<10.95><12><14.4><17.28><20.74><24.88>
	mathx10
}{}
\DeclareFontSubstitution{U}{mathx}{m}{n}

\DeclareSymbolFont{mathxwidebar}{U}{mathx}{m}{n}
\DeclareMathAccent{\widebar}{0}{mathxwidebar}{"73}

\begin{document}

	\title[Operator-based data embedding]{Operator-based data 
		embedding for data-driven control of continuous-time systems from noisy data}
	
	\thispagestyle{plain}
	
	\author{Masashi Wakaiki}
	\address{Graduate School of System Informatics, Kobe University, Nada, Kobe, Hyogo 657-8501, Japan}
	\email{wakaiki@ruby.kobe-u.ac.jp}
	\thanks{This work was supported in part by 
		JSPS KAKENHI Grant Number 24K06866.}
	
	\begin{abstract}
	We propose a data-driven method for 
	designing state-feedback gains that achieve stabilization,
	$H_2$-control, and $H_\infty$-control 
	for continuous-time systems. 
	The state-input data are assumed to be corrupted by 
	process noise, measurement noise, and input disturbances.
	We first characterize the set of 
	all systems consistent with the noisy data 
	using operator-based data embedding.
	This characterization yields necessary and sufficient conditions
	for data informativity under a certain class of noise. 
	These conditions are formulated 
	as linear 
	matrix inequalities, and the feedback gains are 
	constructed from their solutions.
	To enable direct controller design from noisy sampled data
	for continuous-time systems, we also obtain an upper bound on the reconstruction error of continuous-time signals.
	\end{abstract}
	
	
	\maketitle
	
\section{Introduction}
\subsubsection*{Background}
Designing controllers for unknown dynamical systems
is a central problem in control theory.
In model-based control, a mathematical model of the system
is first identified and then used for controller design.
In contrast, data-driven control bypasses explicit system 
identification and directly constructs controllers from
measured data.
This approach is particularly useful 
when constructing an accurate model is difficult or time-consuming.

A fundamental issue in data-driven control is whether the available data contain enough 
information about the unknown system.
One influential line of research
addresses this issue through the notion of 
data informativity~\cite{Waarde2020}.
Rather than identifying a nominal model,
this approach considers the set of all systems that are 
consistent with the data.
It asks whether one can guarantee a desired property for all
data-consistent systems and construct controllers that achieve prescribed control objectives for all such systems.
In this sense, data informativity provides a theoretical framework for characterizing when the available data are informative 
enough for analysis and control.

Since data are typically obtained through sampling,
data informativity has been extensively studied for 
discrete-time systems. Here, 
we review discrete-time results on stabilization, 
$H_2$-control, and $H_\infty$-control, which are directly relevant to this study. 
For other developments, see, e.g., the tutorial paper~\cite{Waarde2023}.
Necessary and sufficient conditions under which
state-input data are  informative 
for state-feedback stabilization were first derived in~\cite{Waarde2020}. 
Data corrupted by process noise were subsequently 
considered in~\cite{Waarde2022TAC,Steentjes2022,Waarde2023SIAM},
where informativity for stabilization was characterized 
in terms of linear matrix inequalities (LMIs). This LMI-based characterization was extended to output-feedback stabilization
via behavioral theory in~\cite{Waarde2023TAC}.
Stabilization from data corrupted by measurement noise and input disturbances 
was also studied in both the state-feedback~\cite{Bisoffi2024} and output-feedback~\cite{Li2026Automatica} cases.
In the setting of performance-guaranteed control,
necessary and sufficient conditions under which
state-input data corrupted by process noise are informative 
for $H_2$-control and $H_\infty$-control
were established in \cite{Waarde2022TAC}. These informativity  results 
were further extended in \cite{Kaminaga2025arXiv} to
the case where process noise,
measurement noise, and input disturbances are present.

Motivated by the observation that many physical systems evolve in continuous time,
data informativity has 
also been investigated in the continuous-time setting. 
To obtain formulations analogous to the discrete-time case, one line of work uses data-consistency conditions based on derivative measurements;
see, e.g., \cite{DePersis2020,Bisoffi2022,Eising2025}.
Since obtaining reliable derivative measurements can be impractical in many real-world applications, 
derivative-free approaches have been developed to 
circumvent this limitation. 
One such approach employs
a polynomial orthogonal basis~\cite{Rapisarda2024},
and another relies on generalized sampling~\cite{Ohta2024MTNS, Ohta2024,Wang2025}.
Both methods transform continuous-time data into 
finite discrete sequences, which may lead to
approximation errors or
information loss.
In contrast, the synthesis-operator approach proposed in \cite{Wakaiki2025Cont} embeds continuous-time data into an integral operator,
which enables informativity analysis without 
discretizing the data.
Based on this data embedding, a necessary and sufficient LMI 
condition
under which state-input data corrupted by process noise are informative for stabilization was obtained in \cite{Wakaiki2025Cont}. The synthesis-operator approach was subsequently applied to output-feedback stabilization from multiple input-output trajectories in \cite{Wakaiki2026IOdata}.

\subsubsection*{Contributions and technical challenges}
In this paper, we investigate the informativity of continuous-time data without relying on derivative measurements.
We assume that noisy state-input data are available.
The noise class considered in this paper 
allows for process noise, measurement noise, and input disturbances, and can be 
regarded as a continuous-time counterpart of 
the data-perturbation class studied in \cite{Kaminaga2025arXiv}. 
Our aim is to characterize
informativity under this noise class 
by extending the synthesis-operator approach 
\cite{Wakaiki2025Cont,Wakaiki2026IOdata}.
to handle measurement noise and input disturbances in addition to process noise.

The contributions of this paper are threefold.
First, we characterize the set of all systems consistent with the given data under the considered noise class
in terms of a quadratic matrix inequality (QMI) involving synthesis operators. Second, based on this characterization, we establish necessary and sufficient LMI conditions 
under which state-input data are informative for stabilization, $H_2$-control, and $H_{\infty}$-control.
Third, from the viewpoint of synthesis operators, 
we provide an upper bound on the error associated with the reconstruction of continuous-time signals from noisy sampled data.
Combining the second and third contributions 
enables the direct design of controllers for 
continuous-time systems from noisy sampled data.

The characterization of the set of all data-consistent systems provides the basis for our data-driven control results.
Indeed, both the LMI conditions for informativity 
and the error analysis for reconstructed 
continuous-time data 
rely on this characterization. 
The main technical difficulty lies in proving 
the necessity of the QMI characterization.
Showing that the QMI provides a sufficient condition for data consistency is relatively straightforward, whereas proving that it is also necessary is involved. This difficulty arises from
the specific structure of synthesis operators.
In the case where only process noise is present, one can easily show that the operator constructed from the QMI has this structure; see \cite{Wakaiki2025Cont}.
However, when measurement noise and input disturbances are also present, 
verifying this property becomes more delicate. 
Resolving this issue is one of the main 
technical contributions of this paper.

\subsubsection*{Comparison with related work}
We finally compare the proposed approach 
with the filtering approach, 
which constitutes another powerful class of 
derivative-free data-driven control methods 
for continuous-time systems. 
In \cite{Possieri2025,Bosso2025}, 
stabilizing controllers were constructed using
state-input data and a low-pass filter.
These design methods were extended in \cite{Bosso2025,Bosso2025arXiv,Gao2025,Bosso2025Noisy,Possieri2026,Li2026}
to the case where only input-output data are available. 
In particular, a data-driven method was proposed in \cite{Bosso2025Noisy}
for designing an observer-based stabilizing
controller when the output data are corrupted by process noise and measurement noise.

Compared with the above related studies, 
the proposed method has the following advantages.
First, in the filtering approach, 
the design conditions for stabilizing controllers depend on the filter parameters, since
the controller is constructed using filtered data.
In contrast, our approach directly embeds continuous-time data into synthesis operators and analyzes the informativity
of the original data.
The resulting necessary and sufficient conditions 
for informativity
do not depend on such design parameters.

Second, while the 
existing LMI-based filtering methods
have mainly focused on stabilization, 
the proposed framework also addresses performance-guaranteed control. 
Specifically, our LMI conditions enable 
the design of controllers not only for stabilization 
but also for $H_2$-control and $H_{\infty}$-control.

Third, the proposed approach can 
also handle the case where only noisy 
sampled data are available. 
This is achieved through the reconstruction-error analysis 
for continuous-time data.
The reconstructed data can be interpreted 
as the original continuous-time data 
corrupted by measurement noise.
Consequently, by incorporating reconstruction-error bounds, our synthesis-operator approach enables 
derivative-free data-driven control for continuous-time systems from noisy sampled data.

\subsubsection*{Organization}
The rest of this paper is structured as follows.
In Section~\ref{sec:data_consistent_systems}, 
we introduce the class of systems and data
considered in this study, and 
then characterize the set of systems consistent with the 
noisy data
in terms of synthesis operators. 
Section~\ref{sec:Data_driven_control}
is devoted to deriving necessary and sufficient LMI conditions under which the noisy data  are informative for stabilization, $H_2$-control, and $H_\infty$-control. 
In Section~\ref{sec:error_analysis}, 
we discuss the error associated with reconstructing continuous-time data from sampled data. In Section~\ref{sec:example}, we demonstrate the proposed method through a numerical example based on an aircraft model. Finally, Section~\ref{sec:conclusion} concludes the paper.

\subsubsection*{Notation}
Let $\mathbb{N}$ denote the set of positive integers.
For $N \in \mathbb{N}$, we use
$(a_k)_{k=1}^N$ to denote an $N$-tuple, where 
$a_1,\dots,a_N$ need not to be in the same set.
Let $\mathbb{S}^n$ denote the set of $n \times n$ 
real symmetric matrices.
Let $\mathbb{S}_+^n$ and 
$\mathbb{S}_{++}^n$
denote the sets of $n \times n$ real nonnegative definite matrices and $n \times n$ real positive definite matrices, respectively.
For $A \in \mathbb{S}^n$, we write $A\succ0$ and $A \succeq 0$ 
if $A$ is positive definite and nonnegative definite, respectively.
Similarly, we write $A\prec0$ and $A \preceq 0$ 
if $A$ is negative definite and nonpositive definite,
respectively.
For a matrix $A$, we denote by $A^{\top}$ and $A^+$ its transpose and Moore--Penrose inverse, respectively.
For $\Theta \in \mathbb{S}_+^n$,
we denote by $\Theta^{1/2}$ the nonnegative
definite square root of $\Theta$.
The $n \times n$ identity matrix and the $n \times m$
zero matrix are denoted by 
$I_n$ and $0_{n \times m}$,
respectively.

Let $a,b \in \mathbb{R}$ with $a < b$. We denote by $\Ls([a,b];\mathbb{R}^n)$
the space of square-integrable measurable functions 
from $[a,b]$ to $\mathbb{R}^n$.
We denote by $\Hso([a,b];\mathbb{R}^n)$
the space of all absolutely continuous functions $\phi\colon
[a,b] \to \mathbb{R}^n$ such that $\phi' \in \Ls([a,b];\mathbb{R}^n)$.
For brevity, we write
$\Ls[a,b] \coloneqq \Ls([a,b];\mathbb{R})$
and 
$\Hso[a,b] \coloneqq \Hso([a,b];\mathbb{R})$.
We define $\Hs[a,b]$ to be the subspace of $\Hso[a,b]$ consisting of all functions
$\phi$ satisfying $\phi(a)=\phi(b)=0$.
The inner product on $\Ls[a,b]$ 
is given by
\[
\langle f, g \rangle_{\Ls} = 
\int_a^b f(t) g(t)dt,\quad f,g \in \Ls[a,b],
\]
and
the inner product on 
$\Hs[a,b]$
is given by
\[
\langle \phi, \psi \rangle_{\Hs} = 
\int_a^b \phi'(t) \psi'(t)dt,\quad \phi,\psi \in \Hs[a,b].
\]

Let $X$ and $Y$ be Hilbert spaces.
We write $\mathcal{L}(X,Y)$ for the space of 
bounded linear operators from $X$ to $Y$.
For $T \in \mathcal{L}(X,Y)$,
its range and kernel 
are denoted by
$\Ran T$ and 
$\Ker T$, respectively.
The Hilbert space adjoint of $T$ is denoted by $T^*$.

\section{Systems consistent with noisy data}
\label{sec:data_consistent_systems}
In this section, we present an operator-theoretic approach
that relates noisy continuous-time data to the underlying system dynamics.
Following a description of the system and the data considered in this paper,
we introduce an operator-based data-embedding method in Section~\ref{sec:data_embedding}.
We then
characterize the set of systems consistent with the noisy data in Section~\ref{sec:characterization_data_consistent_systems}.
\subsection{Operator-based data embedding}
\label{sec:data_embedding}
We consider the continuous-time linear system
\begin{equation}
	\label{eq:system}
	x'(t)  = Ax(t) + Bu(t),\quad t \geq 0,
\end{equation}
where $x(t) \in \mathbb{R}^n$ 
and $u(t) \in \mathbb{R}^m$ are the state and the control input
at time $t \geq 0$, respectively.
Let the set $\Sigma_{n,m}$ of systems
be defined by
\[
\Sigma_{n,m} \coloneqq \{(A,B): A \in \mathbb{R}^{n\times n}
~\text{and}~
B \in \mathbb{R}^{n \times m} \}.
\]
Let $N \in \mathbb{N}$ and $\tau_1,\dots,\tau_{N} >0$.
Suppose that $N$ pairs of state-input
trajectories are available.
Let 
$(x_k,u_k)$ be 
the $k$-th state-input pair, where
$x_k\in \Hso([0,\tau_k];\mathbb{R}^n)$ and $u_k\in \Ls([0,\tau_k];\mathbb{R}^m)$.
We assume that
$x_k'$, $x_k$, and $u_k$ are corrupted by $p_k,q_k \in \Ls([0,\tau_k];\mathbb{R}^n)$ and $r_k \in  \Ls([0,\tau_k];\mathbb{R}^m)$, respectively.
More precisely, we consider the situation where
an unknown true system $(A_s,B_s) \in 
\Sigma_{n,m} $ satisfies
\begin{equation}
	\label{eq:data_consistent}
	(x_k'-p_k) = A_s(x_k-q_k) + B_s(u_k-r_k)
	\quad \text{a.e.~on $[0,\tau_k]$ for all $k=1,\dots,N$}.
\end{equation}
We refer to $(p_k,q_k,r_k)_{k=1}^{N}$ as the noise
in the state-input data $(x_k,u_k)_{k=1}^{N}$. We make the following remark concerning the noise
structure.

\begin{remark}[Noise structure]
	Suppose that the $k$-th 
	state-input pair $(x_k,u_k)$ is obtained from
	the system subject to process noise,
	measurement noise, and input disturbances.
	Let $q_k \in \Hso([0,\tau_k];\mathbb{R}^n)$ denote the additive measurement noise to the state. Then
	the true state derivative is given by 
	$x_k' - q_k'$. In the presence of process noise $v_k \in \Ls([0,\tau_k];\mathbb{R}^n)$ and
	input disturbance $r_k\in \Ls([0,\tau_k];\mathbb{R}^m)$, the 
	state-input data $(x_k,u_k)$ satisfy
	\[
	x_k' - q_k' = A_s(x_k - q_k) + B_s(u_k - r_k) + v_k
	\quad \text{a.e.~on $[0,\tau_k]$}.
	\] 
	The noise term $p_k$ in \eqref{eq:data_consistent} is given by $p_k = q_k' + v_k$. In this study, however, we do not restrict our framework to such a specific noise structure. Instead, we treat $p_k$ and $q_k$ as independent unknown signals in $\Ls([0,\tau_k];\mathbb{R}^n)$ as in
	the discrete-time case~\cite{Bisoffi2024, Kaminaga2025arXiv}.
	\hspace*{\fill} $\triangle$ 
\end{remark}
Our data-embedding method employs synthesis operators on
the $\Hs$-space introduced 
in \cite[Definition~2.1]{Wakaiki2025Cont}.

		\begin{definition}
			\label{def:synthesis_op}
			Let $a,b \in \mathbb{R}$ satisfy $a < b$, and let 
			$f \in \Ls ([a,b]; \mathbb{R}^n)$.
			Define
			the operators $F,F_{\sd} \in \mathcal{L}(\Hs  [a,b], \mathbb{R}^n)$  by
			\begin{align*}
				F\phi \coloneqq \int_a^b  \phi(t)f(t) dt
				\quad \text{and} \quad 
				F_{\sd}\phi  \coloneqq -\int_a^{b}  \phi '(t)f(t) dt 
			\end{align*}
			for $\phi  \in \Hs [a,b]$.
			We call $F$ the {\em synthesis operator associated with $f$}
			and  $F_{\sd}$ the {\em differentiated 
				synthesis operator associated with $f$}.
		\end{definition}
		
		We now define the synthesis operators used in the proposed method, which are associated with the state-input data $(x_k,u_k)$ and 
		the noise $(p_k,q_k,r_k)_{k=1}^{N}$.
		Let $\tau = (\tau_k)_{k=1}^{N}$ with $\tau_1,\dots,\tau_{N} >0$. To simplify the notation,
		we define 
		the data set $\Gamma_{\tau}$ and 
		the noise set $\Delta_{\tau}$ by
		\begin{align*}
			\Gamma_{\tau} &\coloneqq 
			\{ 
			(x_k,u_k)_{k=1}^{N}:
			x_k \in \Hso([0,\tau_k];\mathbb{R}^n)
			~\text{and}~
			u_k \in \Ls([0,\tau_k];\mathbb{R}^m)~\text{for all $k=1,\dots,N$}
			\}, \\
			\Delta_{\tau} &\coloneqq 
			\{ 
			(p_k,q_k,r_k)_{k=1}^{N}:
			p_k,q_k \in \Ls([0,\tau_k];\mathbb{R}^n)
			~\text{and}~
			r_k \in \Ls([0,\tau_k];\mathbb{R}^m)~\text{for all $k=1,\dots,N$}
			\}.
		\end{align*}
		We also denote by
		$\mathcal{H}_{\tau} $
		the orthogonal direct sum
		of $(\Hs[0,\tau_k])_{k=1}^{N}$; that is
		$\Ht$ is  defined by
		\begin{equation*}
			\Ht \coloneqq 
			\{
			(\phi_k)_{k=1}^{N}: \phi_k \in \Hs[0,\tau_k] \text{~for all $k=1,\dots,N$}
			\}
		\end{equation*}
		with the standard inner product $\langle \cdot , \cdot \rangle_{\Ht}$ given by the sum of the 
		componentwise $\Hs$-inner products.
		
		Let $(x_k,u_k)_{k=1}^{N} \in \Gamma_{\tau}$ and 
		$(p_k,q_k,r_k)_{k=1}^{N} \in \Delta_{\tau}$.
		Throughout this paper, we denote by $X_{\sd,k}$
		the differentiated synthesis operator associated with
		$x_k$.
		We also denote by $X_k$,
		$U_k$,
		$P_k$,
		$Q_k$, and 
		$R_k$
		the synthesis operators associated with 
		$x_k$,
		$u_k$,
		$p_k$,
		$q_k$, and 
		$r_k$,
		respectively.
		We define $X_{\sd} \in \mathcal{L}(\Ht, \mathbb{R}^n)$ by
		\[
		X_{\sd} \phi \coloneqq 
		\sum_{k=1}^{N} X_{\sd,k} \phi_k,\quad \phi = 
		(\phi_k)_{k=1}^{N} \in \Ht.
		\]
		Analogous definitions apply to the operators
		$X$, $U$, $P$, $Q$, and $R$.
		Note that integration by parts yields
		\begin{equation}
			\label{eq:Xid_int_by_parts}
			X_{\sd} \phi = 
			\sum_{k=1}^{N} \int_0^{\tau_k} \phi_k(t) x_k'(t) dt
		\end{equation}
		for all $\phi = (\phi_k)_{k=1}^{N}\in \Ht$.
		Finally, we define 
		$Z,W \in \mathcal{L}(\Ht, \mathbb{R}^{2n+m})$ by
		\begin{equation}
			\label{eq:ZW_def}
			Z\phi  \coloneqq  \begin{bmatrix}
				X_{\sd}\phi  \\ -X\phi  \\ -U\phi 
			\end{bmatrix} \quad \text{and} \quad 
			W\phi  \coloneqq \begin{bmatrix}
				P\phi  \\ -Q\phi  \\ -R\phi 
			\end{bmatrix}
			,\quad \phi  \in \Ht.
		\end{equation}
		The state-input data $(x_k,u_k)_{k=1}^{N}$ and the noise
		$(p_k,q_k,r_k)_{k=1}^{N}$ are embedded into the operators $Z$
		and $W$, respectively.
		Although $Z$ and $W$ are operators on the infinite-dimensional
		space $\Ht$, we can regard $ZZ^* $ and $WW^*$ as matrices
		of size $(2n+m) \times (2n+m)$ due to the finite-rank property of $Z$ and $W$.
		We call $ZZ^*$ the {\em data-associated Gramian}
		and $WW^*$ the {\em noise-associated Gramian}.
		To compute the matrices $ZZ^* $ and $WW^*$,
		observe first that
		\begin{align}
			ZZ^* &=  
			\sum_{k=1}^{N} 
			\begin{bmatrix}
				X_{\sd,k}X_{\sd,k}^* & -X_{\sd,k}X_k^* & -X_{\sd,k}U_k^*\\
				-X_kX_{\sd,k}^* & X_kX_k^* & X_kU_k^* \\
				-U_kX_{\sd,k}^* & U_kX_k^* & U_kU_k^*
			\end{bmatrix}, \label{eq:ZZ}\\
			WW^* &= 
			\sum_{k=1}^{N} 
			\begin{bmatrix}
				P_kP_k^* & -P_kQ_k^* & -P_kR_k^*\\
				-Q_kP_k^* & Q_kQ_k^* & Q_kR_k^* \\
				-R_kP_k^* & R_kQ_k^* & R_kR_k^*
			\end{bmatrix}.
			\label{eq:WW}
		\end{align}
		Then
		each block, such as $X_{\sd, k} X_k^*$,
		in \eqref{eq:ZZ} and \eqref{eq:WW} can be computed through 
		its integral representation developed in
		\cite[Lemma~2.3]{Wakaiki2025Cont}.
		Note that the integral representation
		does not depend on an element $\phi \in \Ht$.
		
		We
		characterize the data-consistency condition 
		\eqref{eq:data_consistent} in terms of synthesis operators.
		The following result 
		extends the noise-free characterization
		developed in \cite[Lemma~2.2]{Wakaiki2025Cont}
		to incorporate noise in $\Delta_{\tau}$.
		\begin{lemma}
			\label{lem:synthesis_rep}
			Let $(x_k,u_k)_{k=1}^{N} \in \Gamma_{\tau}$ and 
			$(p_k,q_k,r_k)_{k=1}^{N} \in \Delta_{\tau}$. Then
			the following statements are equivalent for a fixed system
			$(A,B) \in \Sigma_{n,m}$:
			\begin{enumerate}
				\renewcommand{\labelenumi}{\textup{(\roman{enumi})}}
				\item The differential equations in \eqref{eq:data_consistent} hold.
				\item The operator equation $\begin{bmatrix}
					I_n & A & B
				\end{bmatrix} (Z-W) = 
				0$ holds, where the operators
				$Z,W$ are 
				as in
				\eqref{eq:ZW_def}.
			\end{enumerate}
		\end{lemma}
		
		\begin{proof}
			Let $\phi = (\phi_k)_{k=1}^{N}\in \Ht$.
			Using \eqref{eq:Xid_int_by_parts}, we obtain
			\begin{align}
				&\begin{bmatrix}
					I_n & A & B
				\end{bmatrix} Z \phi - 
				\begin{bmatrix}
					I_n & A & B
				\end{bmatrix} W\phi \notag \\
				&\qquad =
				\sum_{k=1}^{N}
				\int_0^{\tau_k} \phi_k(t)
				\big( (x_k'(t) - p_k(t)) - A (x_k(t)-q_k(t)) - 
				B (u_k(t) - r_k(t)) \big)dt.
				\label{eq:IABZ}
			\end{align}
			The implication (i) $\Rightarrow$ (ii) follows immediately from
			\eqref{eq:IABZ}.
			To show the converse implication, 
			let $k \in \{1,\dots, N\}$ be arbitrary, and let $\phi_j = 0$
			if $j \neq k$. Then \eqref{eq:IABZ} yields
			\[
			\int_0^{\tau_k} \phi_k(t)
			\big( (x_k'(t) - p_k(t)) - A (x_k(t)-q_k(t)) - 
			B (u_k(t) - r_k(t)) \big)dt = 0.
			\]
			Using the standard property of test functions (see, e.g., \cite[Proposition 13.2.2]{Tucsnak2009}),
			we obtain
			\[
			x_k'(t) - p_k(t) = A (x_k(t)-q_k(t)) + 
			B (u_k(t) - r_k(t))  
			\]
			for a.e.~$t \in [0,\tau_k]$.
			Since $k \in \{1,\dots, N\}$  is arbitrary, we conclude that 
			\eqref{eq:data_consistent} holds.
		\end{proof}

		\subsection{Characterization of the set of
			data-consistent systems}
		\label{sec:characterization_data_consistent_systems}
		Lemma~\ref{lem:synthesis_rep} shows that
		the noise effect on the data can be quantified 
		in terms of the operator $W$.
		For a noise-intensity matrix
		$\Theta \in \mathbb{S}_+^{2n+m}$, we define the noise class $\Delta_{\tau,\Theta}$ by
		\begin{equation}
			\label{eq:Delta_tau_Theta_def}
			\Delta_{\tau, \Theta}\coloneqq  
			\{
			(p_k,q_k,r_k)_{k=1}^{N} \in \Delta_{\tau}: WW^* \preceq \Theta
			\},
		\end{equation}
		where 
		the noise-associated Gramian $WW^*$ is given by \eqref{eq:WW}.
		For a scalar $c \geq 0$, 
		$WW^* \preceq c I_{2n+m}$ is equivalent to $\|W\| \leq \sqrt{c}$; see the appendix
		for an interpretation of 
		$\|W\|$ from the viewpoint of the frequency of the associated noise $(p_k,q_k,r_k)_{k=1}^{N}$.

		Using the noise class $\Delta_{\tau,\Theta}$, we also 
		define the set $\Sigma_{\mathfrak{D},\Theta}$ of systems by
		\[
		\Sigma_{\mathfrak{D},\Theta} \coloneqq 
		\{ (A,B) \in \Sigma_{n,m}:
		\text{there exists $(p_k,q_k,r_k)_{k=1}^{N} \in \Delta_{\tau,\Theta}$ such that 
			\eqref{eq:data_consistent} holds}
		\}
		\]
		for data $\mathfrak{D} = (x_k,u_k)_{k=1}^{N} \in \Gamma_{\tau}$
		and a noise-intensity matrix $\Theta \in \mathbb{S}_+^{2n+m}$.
		Systems in $\Sigma_{\mathfrak{D},\Theta}$ are 
		consistent with the data $\mathfrak{D}$ corrupted by the noise in the class
		$\Delta_{\tau,\Theta}$.
		We characterize the system set $\Sigma_{\mathfrak{D},\Theta}$
		in terms of a QMI involving the data-associated Gramian.
		To this end, we
		partition $\Theta \in \mathbb{S}_+^{2n+m}$ as
		\begin{equation}
			\label{eq:Theta_partition}
			\Theta
			=
			\begin{bmatrix}
				\Theta_{11} & \Theta_{12} \\
				\Theta_{12}^{\top} & \Theta_{22}
			\end{bmatrix}
			\quad \text{with $\Theta_{11} \in \mathbb{S}^{n} $ and 
				$\Theta_{22} \in \mathbb{S}^{n+m}$}.
		\end{equation}
		It is well known that
		$\Theta \succ 0$ if and only if
		$\Theta_{22} \succ 0$ and $\Theta_{11} - \Theta_{12}\Theta_{22}^{-1} \Theta_{12}^{\top} \succ 0$.
		The following characterization result uses
		a weaker condition than $\Theta \succ 0$, where
		the invertibility of $\Theta_{22}$ is not required.
		\begin{theorem}
			\label{thm:system_cond}
			Let
			$\mathfrak{D} = (x_k,u_k)_{k=1}^{N}
			\in \Gamma_{\tau}$, and suppose that 
			the noise-intensity matrix $\Theta \in \mathbb{S}_+^{2n+m}$
			partitioned as in
			\eqref{eq:Theta_partition} satisfies
			\begin{equation}
				\label{eq:Theta_cond}
				\Theta_{11} - \Theta_{12}\Theta_{22}^+ \Theta_{12}^{\top} \succ0.
			\end{equation}
			Then
			the following statements are equivalent for all $(A,B) \in \Sigma_{n,m}$:
			\begin{enumerate}
				\renewcommand{\labelenumi}{\textup{(\roman{enumi})}}
				\item $(A,B) \in \Sigma_{\mathfrak{D},\Theta}$.
				\item 
				$\begin{bmatrix}
					I_n \\ A^{\top} \\ B^{\top}
				\end{bmatrix}^{\top} 
				(\Theta - ZZ^*)
				\begin{bmatrix}
					I_n \\ A^{\top} \\ B^{\top}
				\end{bmatrix} \succeq 0
				$, where the
				data-associated Gramian
				$ZZ^* \in \mathbb{S}_+^{2n+m}$ is given by
				\eqref{eq:ZZ}.
			\end{enumerate}
		\end{theorem}
		
		The main challenge in Theorem~\ref{thm:system_cond} is to prove the implication (ii) $\Rightarrow$ (i).
		While the proof of this implication is 
		a variant of that for Douglas' lemma (see \cite[Theorem~1]{Douglas1966} and 
		\cite[Proposition~12.1.2]{Tucsnak2009}),
		a technical difficulty arises from
		the specific structure of the operator $W$
		in the definition of $\Delta_{\tau,\Theta}$.
		In our case, $W$ 
		is not an arbitrary bounded operator but
		is constrained to be a sum of synthesis operators.
		To deal with this difficulty, 
		the assumption \eqref{eq:Theta_cond} is used.
		
		Before proceeding to the proof of Theorem~\ref{thm:system_cond},
		we present the following preliminary result, which 
		gives a 
		condition equivalent to \eqref{eq:Theta_cond}.
		\begin{lemma}
			\label{lem:inverse}
			Let $\Theta \in \mathbb{S}_+^{j+\ell}$
			be partitioned 
			as 
			\begin{equation}
				\label{eq:Theta_partition_lemma}
				\Theta
				=
				\begin{bmatrix}
					\Theta_{11} & \Theta_{12} \\
					\Theta_{12}^{\top} & \Theta_{22}
				\end{bmatrix}
				\quad \text{with $\Theta_{11} \in \mathbb{S}^{j} $ and 
					$\Theta_{22} \in \mathbb{S}^{\ell}$}.
			\end{equation}
			Then 
			the following statements are equivalent:
			\begin{enumerate}
				\renewcommand{\labelenumi}{\textup{(\roman{enumi})}}
				\item $\Theta_{11} - \Theta_{12}\Theta_{22}^+ \Theta_{12}^{\top} \succ 0$.
				\item $\Ker \Theta \subseteq \Ran 
				\begin{bmatrix}
					0_{j\times \ell} \\ I_{\ell}
				\end{bmatrix}$.
			\end{enumerate}
		\end{lemma}
		\begin{proof}
			The
			nonnegative definite matrix $\Theta$ partitioned as in
			\eqref{eq:Theta_partition_lemma}
			satisfies the following properties (see, e.g., 
			\cite[Proposition~10.2.5]{Bernstein2018}):
			\begin{align}
				&\Theta_{11} - \Theta_{12}\Theta_{22}^+ \Theta_{12}^{\top} \succeq 0,\label{eq:nn_prop1}\\
				&\Theta_{22}\Theta_{22}^+ \Theta_{12}^{\top} = \Theta_{12}^{\top}. \label{eq:nn_prop2}
			\end{align}

			First, we prove the implication (ii) $\Rightarrow $ (i).
			By \eqref{eq:nn_prop1}, it suffices to show that 
			$\Theta_{11} - \Theta_{12}\Theta_{22}^+ 
			\Theta_{12}^{\top}$ is invertible.
			Let $\mu_1 \in \mathbb{R}^j$ satisfy
			\begin{equation}
				\label{eq:Theta_first_element}
				(\Theta_{11} - \Theta_{12}\Theta_{22}^+ 
				\Theta_{12}^{\top})\mu_1 =0,
			\end{equation}
			and define $\mu_2 \coloneqq -\Theta_{22}^+ \Theta_{12}^{\top}\mu_1$.
			Then
			\begin{equation}
				\label{eq:Theta_mu}
				\Theta 
				\begin{bmatrix}
					\mu_1 \\ \mu_2
				\end{bmatrix} =
				\begin{bmatrix}
					\Theta_{11}\mu_1 - \Theta_{12}\Theta_{22}^+ \Theta_{12}^{\top}\mu_1 \\
					\Theta_{12}^{\top}\mu_1 - \Theta_{22}\Theta_{22}^+ \Theta_{12}^{\top}\mu_1
				\end{bmatrix}.
			\end{equation}
			It follows from \eqref{eq:nn_prop2} that
			\begin{equation}
				\label{eq:Theta_second_element}
				\Theta_{12}^{\top}\mu_1 - \Theta_{22}\Theta_{22}^+ \Theta_{12}^{\top}\mu_1 =
				\Theta_{12}^{\top}\mu_1 -  \Theta_{12}^{\top}\mu_1 = 0.
			\end{equation}
			Substituting \eqref{eq:Theta_first_element}
			and \eqref{eq:Theta_second_element} into
			\eqref{eq:Theta_mu}, we obtain
			\[
			\begin{bmatrix}
				\mu_1 \\ \mu_2
			\end{bmatrix} \in \Ker \Theta .
			\]
			This and statement~(ii) imply that $\mu_1 = 0$.
			Therefore, $\Theta_{11} - \Theta_{12}\Theta_{22}^+ 
			\Theta_{12}^{\top}$ is invertible.
			
			Next, we prove the implication (i) $\Rightarrow $ (ii).
			Let $\mu_1 \in \mathbb{R}^j$ and 
			$\mu_2 \in \mathbb{R}^{\ell}$ satisfy
			\begin{equation}
				\label{eq:KerTheta}
				\begin{bmatrix}
					\mu_1 \\ \mu_2
				\end{bmatrix} \in \Ker \Theta.
			\end{equation}
			It is enough to show that $\mu_1 = 0$.
			Observe first that 
			\eqref{eq:KerTheta} can be written as 
			\begin{equation}
				\label{eq:KerTheta2}
				\Theta_{11}\mu_1 + \Theta_{12}\mu_2 =0
				\quad \text{and} \quad 
				\Theta_{12}^{\top}\mu_1 + \Theta_{22}\mu_2 =0.
			\end{equation}
			From the second equation in \eqref{eq:KerTheta2},
			it follows that 
			\[
			\mu_2 = -\Theta_{22}^+ \Theta_{12}^{\top} \mu_1 + 
			\nu
			\]
			for some $\nu \in \Ker \Theta_{22}$; see, e.g.,
			\cite[Proposition 8.1.9]{Bernstein2018}.
			Substituting this into 
			the first equation in \eqref{eq:KerTheta2},
			we obtain
			\begin{equation}
				\label{eq:Theta_mu1_nu}
				(\Theta_{11} - \Theta_{12}\Theta_{22}^+ \Theta_{12}^{\top})\mu_1 + \Theta_{12}\nu = 0.
			\end{equation}
			Since \eqref{eq:nn_prop2} implies
			\[
			\Ker \Theta_{22} \subseteq \Ker \Theta_{12},
			\]
			we obtain $\Theta_{12}\nu = 0$.
			This and \eqref{eq:Theta_mu1_nu} yield
			$(\Theta_{11} - \Theta_{12}\Theta_{22}^+ \Theta_{12}^{\top})\mu_1 =0$.
			Since $\Theta_{11} - \Theta_{12}\Theta_{22}^+ \Theta_{12}^{\top}$ is invertible by statement~(i), 
			we conclude that $\mu_1 = 0$.
		\end{proof}
		
		We are now ready to prove Theorem~\ref{thm:system_cond}.
		\begin{proof}[Proof of Theorem~\ref{thm:system_cond}.]
			First, we prove the implication (i) $\Rightarrow$ (ii).
			Lemma~\ref{lem:synthesis_rep} shows that
			\begin{align}
				\label{eq:Theta_ZZ_WW}
				\begin{bmatrix}
					I_n \\ A^{\top} \\ B^{\top}
				\end{bmatrix}^{\top} 
				(\Theta 
				-
				ZZ^*)
				\begin{bmatrix}
					I_n \\ A^{\top} \\ B^{\top}
				\end{bmatrix} = 
				\begin{bmatrix}
					I_n \\ A^{\top} \\ B^{\top}
				\end{bmatrix}^{\top} 
				(\Theta 
				-
				WW^*)
				\begin{bmatrix}
					I_n \\ A^{\top} \\ B^{\top}
				\end{bmatrix} .
			\end{align}
			Since $(p_k,q_k,r_k)_{k=1}^{N} \in \Delta_{\tau,\Theta}$
			by statement~(i),
			we obtain
			$\Theta - WW^*
			\succeq 0$. From this inequality and \eqref{eq:Theta_ZZ_WW},
			it follows that statement~(ii) holds.
			
			Next, we prove the implication (ii) $\Rightarrow$ (i).
			To simplify the notation, define $T_{\Theta},T \in \mathbb{R}^{n \times (2n+m)}$ by
			\begin{equation}
				\label{eq:T_Theta_T}
				T_{\Theta} \coloneqq
				\begin{bmatrix}
					I_n & A & B
				\end{bmatrix} \Theta^{1/2}\quad \text{and} \quad 
				T \coloneqq
				\begin{bmatrix}
					I_n & A & B
				\end{bmatrix}.
			\end{equation}
			The proof is divided into four steps.
			In Step~1, we construct an operator 
			$S \in \mathcal{L}(\mathbb{R}^{2n+m}, \Ht)$ satisfying 
			\begin{equation}
				\label{eq:STZ}
				S T_{\Theta}^{\top} \mu = (TZ)^{*} \mu
				\quad \text{for all $\mu \in \mathbb{R}^{n}$}.
			\end{equation}
			In Step~2, we show that 
			\begin{equation}
				\label{eq:SS_bound}
				S^*S \preceq I_{2n+m}.
			\end{equation}
			In Steps~3 and 4, we use these results to complete the proof of the implication (ii) $\Rightarrow$ (i).

			\textit{Step~1.}
			By the definitions  of $T_{\Theta}$ 
			and $T$ in \eqref{eq:T_Theta_T},
			we obtain
			\[
			\begin{bmatrix}
				I_n \\ A^{\top} \\ B^{\top}
			\end{bmatrix}^{\top} 
			(\Theta 
			-
			ZZ^*)
			\begin{bmatrix}
				I_n \\ A^{\top} \\ B^{\top}
			\end{bmatrix} =
			T_{\Theta} T_{\Theta}^{\top} - (TZ)(TZ)^*.
			\]
			From statement~(ii), it follows that
			\[
			(TZ)(TZ)^* \preceq T_{\Theta}T_{\Theta}^{\top},
			\]
			which implies that 
			\begin{equation}
				\label{eq:TZ_TT_estimate}
				\|(TZ)^* \mu\|_{\Ht} \leq \|T_{\Theta}^{\top} \mu\|_{\mathbb{R}^{2n+m}}
			\end{equation}
			for all $\mu  \in \mathbb{R}^{n}$.
			Define the bounded linear operator 
			$S_0 \colon \Ran T_{\Theta}^{\top} \to 
			\Ran \,(TZ)^*$ by
			\begin{equation}
				\label{eq:S0_def}
				S_0 T_{\Theta}^{\top} \mu = (TZ)^* \mu,\quad \mu \in \mathbb{R}^{n}.
			\end{equation}
			Then $S_0$ is well-defined. Indeed, if 
			$\mu_1,\mu_2 \in \mathbb{R}^n$ satisfy
			$T_{\Theta}^{\top} \mu_1 = T_{\Theta}^{\top} \mu_2$, then
			the inequality \eqref{eq:TZ_TT_estimate} implies that
			\[
			\|
			(TZ)^* \mu_1 - (TZ)^* \mu_2
			\|_{\Ht}  = 
			\|
			(TZ)^* (\mu_1 -  \mu_2)
			\|_{\Ht} 
			\leq 
			\|T_{\Theta}^{\top} (\mu_1-\mu_2)\|_{\mathbb{R}^{2n+m}}  = 0.
			\]
			Define $S \in \mathcal{L}(\mathbb{R}^{2n+m}, \Ht)$ by
			\begin{equation}
				\label{eq:S_def}
				S \coloneqq S_0 \Pi ,
			\end{equation}
			where $\Pi $ is the the orthogonal projection
			onto $\Ran T_{\Theta}^{\top}$.
			Since
			$
			S T_{\Theta}^{\top} \mu =
			S_0 T_{\Theta}^{\top} \mu 
			$
			for all $\mu \in \mathbb{R}^{n}$,
			we deduce from \eqref{eq:S0_def} that 
			\eqref{eq:STZ} holds.

			\textit{Step~2.}
			Let $\eta  \in \mathbb{R}^{2n+m}$  be arbitrary, and define $\eta_1  \in \mathbb{R}^{2n+m}$ by
			\begin{equation}
				\label{eq:eta1_def}
				\eta_1 \coloneqq  \Pi  \eta.
			\end{equation}
			There exists $\mu_1 \in \mathbb{R}^{n}$ such that 
			\begin{equation}
				\label{eq:mu1_def}
				\eta_1 = T_{\Theta}^{\top} \mu_1.
			\end{equation}
			From \eqref{eq:S0_def}--\eqref{eq:mu1_def}, it follows that
			\begin{equation}
				\label{eq:S_eta_TZ}
				S\eta = 
				S_0 \eta_1 = S_0T_{\Theta}^{\top} \mu_1 = 
				(TZ)^{*} \mu_1.
			\end{equation}
			Combining this with
			\eqref{eq:TZ_TT_estimate}, we obtain
			\begin{align*}
				\|S\eta\|_{\Ht}  &=
				\|
				(TZ)^* \mu_1
				\|_{\Ht}  \\
				&\leq
				\|
				T_{\Theta}^{\top} \mu_1
				\|_{\mathbb{R}^{2n+m}}  \\
				&= \|\eta_1\|_{\mathbb{R}^{2n+m}} \\
				&\leq \|\eta\|_{\mathbb{R}^{2n+m}}. 
			\end{align*}
			Since $\eta \in \mathbb{R}^{2n+m}$ is arbitrary, we obtain 
			$\|S\| \leq 1$. Hence,
			\eqref{eq:SS_bound} holds.
			
			\textit{Step~3.}
			Define $S_{\Theta} \in \mathcal{L}(\mathbb{R}^{2n+m}, \Ht)$ by
			\[
			S_{\Theta} \coloneqq S \Theta^{1/2}.
			\]
			By \eqref{eq:SS_bound}, we have
			\begin{equation}
				\label{eq:noise_bound_prop}
				S_{\Theta}^* S_{\Theta} = 
				\Theta^{1/2} S^*S \Theta^{1/2} \preceq \Theta.
			\end{equation}
			Moreover, \eqref{eq:STZ} yields
			\[
			S_{\Theta} T^{\top}\mu = S \Theta^{1/2} T^{\top}\mu =
			S T_{\Theta}^{\top}\mu = (TZ)^* \mu
			\]
			for all $\mu \in \mathbb{R}^n$. This implies that 
			\begin{equation}
				\label{eq:data_consistency_prop}
				TS_{\Theta}^* = TZ.
			\end{equation}
			Therefore,
			the statement~(i) is proved
			if we can show that the synthesis operators $(P_k,Q_k,R_k)_{k=1}^{N}$ 
			associated with some
			$(p_k,q_k,r_k)_{k=1}^{N} \in \Delta_{\tau}$ satisfy
			\begin{equation}
				\label{eq:S_theta_rep}
				S_{\Theta}^* \phi = 
				\mathlarger{\sum}_{k=1}^{N}
				\begin{bmatrix}
					P_k \phi_k\\ -Q_k\phi_k \\ -R_k\phi_k
				\end{bmatrix}
			\end{equation}
			for all $\phi = (\phi_k)_{k=1}^{N} \in \Ht$.
			Indeed,
			$(p_k,q_k,r_k)_{k=1}^{N} \in \Delta_{\tau,\Theta}$
			follows from
			\eqref{eq:noise_bound_prop}, and
			\eqref{eq:data_consistent} holds by
			\eqref{eq:data_consistency_prop} and 
			Lemma~\ref{lem:synthesis_rep}.
			
			\textit{Step~4.}
			It remains to prove that \eqref{eq:S_theta_rep} 
			holds for some $(p_k,q_k,r_k)_{k=1}^{N} \in \Delta_{\tau}$.
			To this end, 
			we show that the adjoint $S^*$ can be written as 
			\begin{equation}
				\label{eq:adjoint_S_rep}
				S^* = \Pi T_{\Theta}^+  TZ .
			\end{equation}
			
			Let $\eta \in \mathbb{R}^{2n+m}$ and $\phi \in \Ht$ be arbitrary.
			Let $\eta_1\in \mathbb{R}^{2n+m}$ and $\mu_1 \in \mathbb{R}^{n}$ 
			be as in \eqref{eq:eta1_def} and 
			\eqref{eq:mu1_def}, respectively.
			Then \eqref{eq:S_eta_TZ} gives
			\begin{equation}
				\label{eq:S_inner_product}
				\langle S\eta ,\phi 
				\rangle_{\Ht} = \langle \mu_1, TZ\phi 
				\rangle_{\mathbb{R}^n}.
			\end{equation}
			Since $(T_{\Theta}^+)^{\top} = (T_{\Theta}^{\top})^+$,
			we deduce from \eqref{eq:mu1_def} that 
			\begin{equation}
				\label{eq:T_eta_mu}
				\mu_1 = (T_{\Theta}^+)^{\top}\eta_1 + \nu,
			\end{equation}
			for some $\nu \in \Ker \,(T_\Theta^{\top})$; see, e.g.,
			\cite[Proposition 8.1.9]{Bernstein2018}.
			From \eqref{eq:S_inner_product} and \eqref{eq:T_eta_mu},
			it follows that
			\begin{align}
				\label{eq:Seta_eta1_nu}
				\langle S\eta,\phi 
				\rangle_{\Ht} 
				=
				\langle 
				\eta_1, T_{\Theta}^+ TZ \phi 
				\rangle_{\mathbb{R}^n} + 
				\langle 
				\nu, TZ \phi 
				\rangle_{\mathbb{R}^n} .
			\end{align}
			Since $\Theta T^{\top} \nu = 0$, 
			Lemma~\ref{lem:inverse} together with \eqref{eq:Theta_cond} 
			shows that 
			\[
			T^{\top} \nu = 
			\begin{bmatrix}
				I_n \\ A^{\top} \\ B^{\top}
			\end{bmatrix}\nu \in \Ran 
			\begin{bmatrix}
				0_{n\times (n+m)} \\ I_{n+m}
			\end{bmatrix}.
			\] 
			This implies that $\nu = 0$.
			By \eqref{eq:Seta_eta1_nu} and 
			\eqref{eq:eta1_def}, we therefore have
			\begin{align*}	
				\langle S\eta,\phi 
				\rangle_{\Ht} &=
				\langle 
				\eta_1, T_{\Theta}^+ TZ \phi 
				\rangle_{\mathbb{R}^{2n+m}} \\
				&= \langle 
				\Pi  \eta, T_{\Theta}^+  TZ \phi 
				\rangle_{\mathbb{R}^{2n+m}} \\
				&= \langle 
				\eta, \Pi T_{\Theta}^+  TZ \phi 
				\rangle_{\mathbb{R}^{2n+m}}.
			\end{align*}
			Since 
			$\eta \in \mathbb{R}^{2n+m}$ and $\phi \in \Ht$ are arbitrary, \eqref{eq:adjoint_S_rep} holds.
			
			Using \eqref{eq:Xid_int_by_parts},  we obtain
			\[
			Z\phi = 
			\mathlarger{\sum}_{k=1}^{N}
			\mathlarger{\int}_0^{\tau_k} \phi_k (t)
			\begin{bmatrix}
				x_k'(t) \\ -x_k(t) \\ -u_k(t)
			\end{bmatrix} dt
			\]
			for all $\phi = (\phi_k)_{k=1}^{N} \in \Ht$.
			By \eqref{eq:adjoint_S_rep}, we obtain
			\[
			S_{\Theta}^* \phi = 
			\mathlarger{\sum}_{k=1}^{N}
			\mathlarger{\int}_0^{\tau_k} \phi_k (t)
			\Theta^{1/2}\Pi T_{\Theta}^+  T
			\begin{bmatrix}
				x_k'(t) \\ -x_k(t) \\ -u_k(t)
			\end{bmatrix} dt.
			\]
			Since
			\[
			\mathlarger{\sum}_{k=1}^{N}
			\begin{bmatrix}
				P_k \phi_k\\ -Q_k\phi_k \\ -R_k\phi_k
			\end{bmatrix}
			=
			\mathlarger{\sum}_{k=1}^{N}
			\mathlarger{\int}_0^{\tau_k} \phi_k (t)
			\begin{bmatrix}
				p_k(t) \\ -q_k(t) \\ -r_k(t)
			\end{bmatrix} dt,
			\]
			the assertion
			\eqref{eq:S_theta_rep} holds if we define
			$(p_k,q_k,r_k)_{k=1}^{N} \in \Delta_{\tau}$
			by
			\[
			\begin{bmatrix}
				p_k(t) \\ q_k(t) \\ r_k(t)
			\end{bmatrix} \coloneqq 
			\begin{bmatrix}
				I_n & 0_{n\times n} & 0_{n\times m} \\
				0_{n\times n} & -I_n & 0_{n\times m} \\
				0_{m\times n} & 0_{m\times n} & -I_m
			\end{bmatrix} 
			\Theta^{1/2}\Pi T_{\Theta}^+  T
			\begin{bmatrix}
				x_k'(t) \\ -x_k(t) \\ -u_k(t)
			\end{bmatrix}
			\]
			for $t \in [0,\tau_k]$ and $k=1,\dots,N$.
		\end{proof}

		\section{Data-driven control}
		\label{sec:Data_driven_control}
		This section is devoted to developing data-driven control methods based on the characterization of
		the set of all data-consistent systems.
		First, we state the standing assumption for this section and verify that the
		hypotheses required to apply
		the strict matrix $S$-lemma hold in our setting.
		Second, we establish a necessary and sufficient LMI
		condition under which the noisy data are informative for
		quadratic stabilization.
		Finally, we extend this result to 
		$H_2$-control and $H_{\infty}$-control.
		
		We make the following assumption on
		the state-input data $(x_k,u_k)_{k=1}^{N} \in \Gamma_{\tau}$ and
		the noise-intensity matrix $\Theta \in \mathbb{S}_+^{2n+m}$.
		\begin{assumption}
			\label{assump:Theta}
			For $k=1,\dots,N$,
			let  $X_k$ and
			$U_k$ be
			the synthesis operators associated with 
			the $k$-th state
			$x_k \in \Hso([0,\tau_k];\mathbb{R}^n)$ and the $k$-th input $u_k \in \Ls([0,\tau_k];\mathbb{R}^m)$, respectively. 
			Let $\Theta \in \mathbb{S}_+^{2n+m}$ be
			partitioned as in
			\eqref{eq:Theta_partition}.
			Then
			\begin{enumerate}
				\renewcommand{\labelenumi}{\textup{\alph{enumi})}}
				\item There exist a system
				$(A_s,B_s) \in \Sigma_{n,m}$ and
				noise sequences
				$(p_k,q_k,r_k)_{k=1}^{N} \in \Delta_{\tau,\Theta}$ such that
				for all $k=1,\dots,N$
				\[
				x_k'-p_k = A_s(x_k-q_k) + B_s(u_k-r_k)
				\quad \text{a.e.~on $[0,\tau_k]$}.
				\]
				\item $\displaystyle 
				\Theta_{22} \prec \sum_{k=1}^{N} \begin{bmatrix}
					X_k X_k^* & X_k  U_k^* \\
					U_k X_k^* & U_k  U_k^* 
				\end{bmatrix}.
				$
				\item $\Theta_{11} - \Theta_{12}\Theta_{22}^+ \Theta_{12}^{\top} \succ 0$.
			\end{enumerate}
		\end{assumption}
		
		Condition~a) implies that 
		the data $(x_k,u_k)_{k=1}^{N} $
		are generated from the true system $(A_s,B_s)$ subject to
		some noise in the class $\Delta_{\tau,\Theta}$.
		Condition~b) requires the data $(x_k,u_k)_{k=1}^{N} $ to be large compared with
		the noise $(q_k,r_k)_{k=1}^{N}$;
		see also the definition 
		of the noise class $\Delta_{\tau,\Theta}$ in
		\eqref{eq:Delta_tau_Theta_def}.
		A similar condition was considered for discrete-time systems in \cite[Assumption~1]{Bisoffi2024}.
		In the noise-free case when
		$\Theta$ is the zero matrix, condition~b) is equivalent to
		informativity for system identification; see \cite[Proposition~3.2]{Wakaiki2025Cont}
		and \cite[Proposition~5.1]{Wakaiki2026IOdata}.
		Condition~c) is a technical requirement to apply Theorem~\ref{thm:system_cond}.
		
		To characterize data informativity,
		we employ the {\em strict matrix $S$-lemma} 
		presented in \cite[Theorem~4.10]{Waarde2023SIAM}.
		We introduce the matrix partitions and the definitions
		required for its statement.
		Let $\mathcal{M},\mathcal{N}  \in 
		\mathbb{S}^{j+\ell}$ be given. We partition these matrices as 
		\begin{align}
			\label{eq:M_partition}
			\mathcal{M} &=
			\begin{bmatrix}
				\mathcal{M}_{11} & \mathcal{M}_{12} \\
				\mathcal{M}_{12}^{\top} & \mathcal{M}_{22}
			\end{bmatrix}\quad \text{with $\mathcal{M}_{11} \in \mathbb{S}^{j} $ and 
				$\mathcal{M}_{22}\in \mathbb{S}^{\ell}$},\\
			\label{eq:N_partition}
			\mathcal{N} &=
			\begin{bmatrix}
				\mathcal{N}_{11} & \mathcal{N}_{12} \\
				\mathcal{N}_{12}^{\top} & \mathcal{N}_{22}
			\end{bmatrix}\quad \text{with $\mathcal{N}_{11} \in \mathbb{S}^{j} $ and 
				$\mathcal{N}_{22}\in \mathbb{S}^{\ell}$}.
		\end{align}
		Using these partitions,
		we define
		the matrix sets $\mathcal{Z}_{j,\ell}(\mathcal{N})$
		and $\mathcal{Z}_{j,\ell}^+(\mathcal{M})$
		by
		\begin{align}
			\label{eq:ZN_def}
			\mathcal{Z}_{j,\ell}(\mathcal{N}) &\coloneqq 
			\left\{
			E \in \mathbb{R}^{\ell \times j}:
			\begin{bmatrix}
				I_j \\ E
			\end{bmatrix}^{\top}
			\mathcal{N}
			\begin{bmatrix}
				I_j \\E
			\end{bmatrix} \succeq  0
			\right\}, \\
			\label{eq:ZM_def}
			\mathcal{Z}_{j,\ell}^+(\mathcal{M}) &\coloneqq 
			\left\{
			E \in \mathbb{R}^{\ell \times j} :
			\begin{bmatrix}
				I_j \\ E
			\end{bmatrix}^{\top}
			\mathcal{M}
			\begin{bmatrix}
				I_j \\ E
			\end{bmatrix}\succ 0
			\right\}.
		\end{align}
		The strict matrix $S$-lemma is stated as follows.
		\begin{lemma}
			\label{lem:S_lemma}
			Let $\mathcal{M},\mathcal{N} \in 
			\mathbb{S}^{j+\ell}$ be partitioned as in \eqref{eq:M_partition}
			and \eqref{eq:N_partition}, respectively. Assume that $\mathcal{N}$ satisfies
			\begin{equation}
				\label{eq:S_Lemma_cond}
				\mathcal{N}_{22} \prec 0\quad \text{and}
				\quad 
				\mathcal{N}_{11} - 
				\mathcal{N}_{12}\mathcal{N}_{22}^{-1}\mathcal{N}_{12}^{\top}  \succeq  0.
			\end{equation}
			Then $\mathcal{Z}_{j,\ell}(\mathcal{N}) \subseteq
			\mathcal{Z}_{j,\ell}^{+}(\mathcal{M})$
			if and only if 
			there exists a scalar $\alpha \geq 0$
			such that
			$\mathcal{M} - \alpha \mathcal{N} \succ 0$.
		\end{lemma}
		
		To apply the strict matrix $S$-lemma, we show that 
		the conditions given in 
		\eqref{eq:S_Lemma_cond} hold under 
		Assumption~\ref{assump:Theta}.
		In our setting, the matrix
		$\mathcal{N} \in 
		\mathbb{S}^{2n+m}$ is defined by
		\begin{align}
			\label{eq:N_def}
			\mathcal{N} &\coloneqq
			\Theta - ZZ^*,
		\end{align}
		where the data-associated Gramian $ZZ^*\in 
		\mathbb{S}_+^{2n+m}$ is given by \eqref{eq:ZZ}.
		\begin{lemma}
			\label{lem:nec_cond_for_N}
			Suppose that the data
			$\mathfrak{D} = (x_k,u_k)_{k=1}^{N}
			\in \Gamma_{\tau}$ and the noise-intensity matrix $\Theta \in \mathbb{S}_+^{2n+m}$
			satisfy
			Assumption~\ref{assump:Theta}.
			Let $\mathcal{N} \in \mathbb{S}^{2n+m}$ 
			be defined by \eqref{eq:N_def}
			and be partitioned as in \eqref{eq:N_partition} with
			$j=n$ and $\ell = n+m$.
			Then the conditions given in \eqref{eq:S_Lemma_cond} hold.
		\end{lemma}
		\begin{proof}
			Assumption~\ref{assump:Theta}.b) implies that 
			$\mathcal{N}_{22} \prec 0$. 
			By Assumption~\ref{assump:Theta}.a),
			we have
			$(A_s,B_s) \in \Sigma_{\mathfrak{D},\Theta}$.
			Theorem~\ref{thm:system_cond} shows that
			\begin{equation}
				\label{eq:IABN}
				\begin{bmatrix}
					I_n \\ A_s^{\top} \\ B_s^{\top}
				\end{bmatrix}^{\top} 
				\mathcal{N} 
				\begin{bmatrix}
					I_n \\ A_s^{\top} \\ B_s^{\top}
				\end{bmatrix}  \succeq 0.
			\end{equation}
			Using \cite[Fact~8.9.7]{Bernstein2018}, we obtain
			\begin{equation}
				\label{eq:N_Schur}
				\mathcal{N}_{11} - \mathcal{N}_{12}
				\mathcal{N}_{22}^{-1}  \mathcal{N}_{12}^{\top} =
				\begin{bmatrix}
					I_n \\ A_s^{\top} \\ B_s^{\top}
				\end{bmatrix}^{\top} 
				\mathcal{N} 
				\begin{bmatrix}
					I_n \\ A_s^{\top} \\ B_s^{\top}
				\end{bmatrix} -
				\left(
				\begin{bmatrix}
					A_s^{\top} \\ B_s^{\top}
				\end{bmatrix} 
				+
				\mathcal{N}_{22}^{-1} \mathcal{N}_{12}^{\top}
				\right)^{\top}
				\mathcal{N}_{22}
				\left(
				\begin{bmatrix}
					A_s^{\top} \\ B_s^{\top}
				\end{bmatrix} 
				+
				\mathcal{N}_{22}^{-1} \mathcal{N}_{12}^{\top}
				\right).
			\end{equation}
			By \eqref{eq:IABN}, \eqref{eq:N_Schur},
			and	
			$\mathcal{N}_{22}\prec 0$, we conclude that
			$\mathcal{N}_{11} - \mathcal{N}_{12}
			\mathcal{N}_{22}^{-1}  \mathcal{N}_{12}^{\top} \succeq 0$.
		\end{proof}
		
		\subsection{Quadratic stabilization}
		We study the problem of stabilizing
		the system \eqref{eq:system} 
		using the state-feedback law  with a gain 
		$K \in \mathbb{R}^{m \times n}$:
		\begin{equation}
			\label{eq:state_feedback}
			u(t) = Kx(t),\quad t \geq 0.
		\end{equation}
		
		We begin by recalling the definition of data informativity for 
		quadratic stabilization.
		This definition is based on a Lyapunov inequality and 
		was originally introduced for 
		discrete-time systems in \cite[Definition~3]{Waarde2022TAC}.
		\begin{definition}
			The data $\mathfrak{D} = (x_k,u_k)_{k=1}^{N} \in \Gamma_{\tau}$ are {\em 
				informative for
				quadratic stabilization under the noise class $\Delta_{\tau,\Theta}$} if there exist matrices $\Phi \in 
			\mathbb{S}_{++}^n$ and $K \in \mathbb{R}^{m \times n}$ 
			such that
			\begin{equation}
				\label{eq:Lyap}
				(A+BK)\Phi  + \Phi (A+BK)^{\top} \prec 0
			\end{equation}
			for all $(A,B) \in \Sigma_{\mathfrak{D},\Theta}$.
		\end{definition}
		
		The following theorem gives a necessary and sufficient
		LMI condition for 
		this informativity property.
		\begin{theorem}
			\label{thm:stabilization}
			Suppose that the data
			$\mathfrak{D} = (x_k,u_k)_{k=1}^{N}
			\in \Gamma_{\tau}$ and the noise-intensity matrix $\Theta \in \mathbb{S}_+^{2n+m}$
			satisfy Assumption~\ref{assump:Theta}. 
			Then the following statements are equivalent:
			\begin{enumerate}
				\renewcommand{\labelenumi}{\textup{(\roman{enumi})}}
				\item The data $\mathfrak{D}$ are informative for quadratic stabilization under the noise class $\Delta_{\tau,\Theta}$.
				\item There exist matrices $\Phi \in 
				\mathbb{S}_{++}^n$ and $L \in \mathbb{R}^{m \times n}$ 
				such that
			\begin{align}
				\begin{bmatrix}
					0_{n\times n} & -\Phi & -L^{\top} \\
					-\Phi & 0_{n\times n} & 0_{n\times m} \\
					-L & 0_{m\times n} & 0_{m\times m} 
				\end{bmatrix} + ZZ^* - \Theta
				\succ 0,
				\label{eq:stabilization_LMI}
			\end{align}
			where the data-associated Gramian
			$ZZ^* \in \mathbb{S}_+^{2n+m}$ is given by
			\eqref{eq:ZZ}.
		\end{enumerate}
		Moreover, if statement (ii) holds, then
		$K \coloneqq L \Phi^{-1}$ is such that 
		$A+BK$ is Hurwitz for
		all $(A,B) \in \Sigma_{\mathfrak{D},\Theta}$.
	\end{theorem}
	
	\begin{proof}
		First, we prove the implication (i) $\Rightarrow $ (ii).
		Let $\Phi \in 
		\mathbb{S}_{++}^n$ and $K \in \mathbb{R}^{m \times n}$ 
		satisfy the Lyapunov inequality \eqref{eq:Lyap}
		for all $(A,B) \in \Sigma_{\mathfrak{D},\Theta}$.
		Define $\mathcal{M} \in \mathbb{S}^{2n+m}$ by
		\begin{equation}
			\label{eq:M_def}
			\mathcal{M} \coloneqq 
			\begin{bmatrix}
				0_{n\times n} & -\Phi & -\Phi K^{\top} \\
				-\Phi & 0_{n\times n} & 0_{n\times m} \\
				-K\Phi & 0_{m\times n} & 0_{m\times m} 
			\end{bmatrix}.
		\end{equation}
		Since
		\begin{equation}
			\label{eq:M_Lyap}
			\begin{bmatrix}
				I_n \\ A^{\top} \\ B^{\top}
			\end{bmatrix}^{\top} 
			\mathcal{M} 
			\begin{bmatrix}
				I_n \\ A^{\top} \\ B^{\top}
			\end{bmatrix} = 
			-(A+BK)\Phi  - \Phi (A+BK)^{\top}
		\end{equation}
		for all $(A,B) \in \Sigma_{n,m}$,
		the Lyapunov inequality yields
		\[
		\begin{bmatrix}
			A^{\top} \\ B^{\top}
		\end{bmatrix} \in 
		\mathcal{Z}_{n,n+m}^+(\mathcal{M})
		\]
		for all $(A,B) \in \Sigma_{\mathfrak{D},\Theta}$.
		By Theorem~\ref{thm:system_cond}, 
		$(A,B) \in \Sigma_{\mathfrak{D},\Theta}$
		is equivalent to
		\[
		\begin{bmatrix}
			A^{\top} \\ B^{\top}
		\end{bmatrix} \in 
		\mathcal{Z}_{n,n+m}(\mathcal{N}),
		\]
		where $\mathcal{N}\in \mathbb{S}^{2n+m}$ is as in \eqref{eq:N_def}.
		Therefore,
		\[
		\mathcal{Z}_{n,n+m}(\mathcal{N}) \subseteq
		\mathcal{Z}_{n,n+m}^{+}(\mathcal{M}).
		\]
		By Lemma~\ref{lem:nec_cond_for_N} and 
		the strict matrix $S$-lemma,
		there exists a scalar $\alpha \geq 0$ 
		such that $\mathcal{M} - \alpha \mathcal{N} \succ 0$.
		Since the $(2,2)$-block of $\mathcal{M}$ is the zero matrix,
		it follows that
		$\alpha  >0$. Replacing $\Phi/\alpha$
		by $\Phi$
		and $K\Phi/\alpha$ by $L$, we obtain
		the LMI \eqref{eq:stabilization_LMI}.
		
		Next, we prove the implication (ii) $\Rightarrow $ (i).
		Let $\Phi \in 
		\mathbb{S}_{++}^n$ and $L \in \mathbb{R}^{m \times n}$ 
		satisfy the LMI~\eqref{eq:stabilization_LMI}.
		Set $K \coloneqq L\Phi^{-1}$ and define
		$\mathcal{M}$ by \eqref{eq:M_def}.
		Then the LMI~\eqref{eq:stabilization_LMI}
		yields $\mathcal{M}- \mathcal{N} \succ 0$,
		where $\mathcal{N}\in \mathbb{S}^{2n+m}$ is defined by \eqref{eq:N_def}.
		From Theorem~\ref{thm:system_cond},
		it  follows that for all $(A,B) \in \Sigma_{\mathfrak{D},\Theta}$,
		\[
		\begin{bmatrix}
			I_n \\ A^{\top} \\ B^{\top}
		\end{bmatrix}^{\top}  \mathcal{N} \begin{bmatrix}
			I_n \\ A^{\top} \\ B^{\top}
		\end{bmatrix} 
		\succeq 0.
		\]
		Since \eqref{eq:M_Lyap} remains valid for all $(A,B) \in \Sigma_{n,m}$ without assuming that
		statement~(i) holds,
		we have
		\[
		0 \prec \begin{bmatrix}
			I_n \\ A^{\top} \\ B^{\top}
		\end{bmatrix}^{\top} (\mathcal{M} - \mathcal{N}) \begin{bmatrix}
			I_n \\ A^{\top} \\ B^{\top}
		\end{bmatrix} \preceq
		-(A+BK)\Phi  - \Phi (A+BK)^{\top}
		\]
		for all $(A,B) \in \Sigma_{\mathfrak{D},\Theta}$.
		Thus, the data $\mathfrak{D}$ are informative for
		quadratic
		stabilization under the noise class $\Delta_{\tau,\Theta}$, and 
		$A+BK$ is Hurwitz for
		all $(A,B) \in \Sigma_{\mathfrak{D},\Theta}$.
	\end{proof}
	
	\subsection{Performance-guaranteed control}
	\label{sec:H2_control}
	We extend the stabilization result
	to address the problem of 
	$H_2$-control and $H_{\infty}$-control.
	Consider the system
	\begin{equation}
		\label{eq:system_with_performance}
		x'(t)=A(x(t)+q(t))+B(u(t)+r(t))+p(t)\quad t\geq 0,
	\end{equation}
	where $p(t),q(t) \in \mathbb{R}^n$
	and $r(t) \in \mathbb{R}^m$
	are the external signals at time $t \geq 0$.
	Here $p$ corresponds to
	the process noise, $q$ to the state bias (see, e.g., \cite{Hockerdal2009}), and $r$ to the input disturbance.
	To characterize these external signals,
	we introduce the normalized signal 
	$w(t) \in \mathbb{R}^{n_w}$ for $H_2$-control and $H_{\infty}$-control.
	Then $p$, $q$, and $r$ are modeled as 
	\begin{equation}
		\label{eq:external_input}
		\begin{bmatrix}
			p(t) \\ q(t) \\ r(t) 
		\end{bmatrix} = 
		\Omega w(t),\quad  t \geq 0,
	\end{equation}
	where 
	$\Omega\in \mathbb{R}^{(2n+m) \times n_w}$ is a
	weighting matrix used to scale the external signals.
	We design the matrix $\Omega$ by
	using known upper bounds for $p$, $q$, and $r$; see also
	Remarks~\ref{rem:H2_Sigma} and 
	\ref{rem:Hinf_Sigma} below.
	For performance evaluation, we define
	the performance output $z(t) \in \mathbb{R}^{n_z}$ by
	\begin{equation}
		\label{eq:performance_output}
		z(t) = Cx(t) + Du(t) + E w(t),\quad t \geq 0,
	\end{equation}
	where 
	$C \in \mathbb{R}^{n_z\times n}$, 
	$D\in \mathbb{R}^{n_z\times m}$, and
	$E \in \mathbb{R}^{n_z\times n_w}$.
	We assume that the system 
	$(A,B) \in \Sigma_{n,m}$ is unknown, whereas
	the weighting matrix $\Omega$ and 
	the matrices $C$, $D$, and $E$ characterizing the performance
	are known.
	These known matrices are fixed throughout Section~\ref{sec:H2_control}.
	
	For the state-feedback law in \eqref{eq:state_feedback},
	the 
	transfer function $G_K$ from $w$ to $z$ is given by
	\begin{equation}
		\label{eq:transfer_function}
		G_K(s) = (C+DK)(sI_n-(A+BK))^{-1}B_{\Omega}  +E,
	\end{equation}
	where $B_{\Omega} \coloneqq \begin{bmatrix}
		I_n & A & B 
	\end{bmatrix} \Omega$.
	When $A+BK$ is Hurwitz and $E = 0$, we denote the $H_2$-norm of $G_K$ by
	$\|G_K\|_{H_2}$.
	When $A+BK$ is Hurwitz, we denote  the $H_\infty$-norm of $G_K$ by
	$\|G_K\|_{H_\infty}$.
	
	\subsubsection{$H_2$-control}
	In the $H_2$-control setting, we assume that 
	$E = 0$.
	For fixed 
	$\gamma >0$ and 
	$K\in \mathbb{R}^{m \times n}$, 
	the following statements are equivalent (see, e.g., \cite[Section~5.3]{Duan2013}):
	\begin{enumerate}
		\renewcommand{\labelenumi}{\textup{(\roman{enumi})}}
		\item 
		$A+BK$ is Hurwitz and $\|G_K\|_{H_2} < \gamma$.
		\item There exist matrices $\Phi\in \mathbb{S}_{++}^n$ 
		and $\Psi\in \mathbb{S}_{++}^{n_z}$ such that 
		\begin{align}
			&(A+BK)\Phi+\Phi(A+BK)^{\top} + B_{\Omega}B_{\Omega}^{\top}
			\prec 0,\label{eq:H2_LMIs1}\\ 
			&\Psi - (C+DK)\Phi(C+DK)^{\top} \succ 0,\label{eq:H2_LMIs2}\\
			&\trace\,(\Psi) < \gamma^2.
			\label{eq:H2_LMIs3}
		\end{align}
	\end{enumerate}
	
	Motivated by the
	above equivalence, we define the notion of
	informativity for $H_2$-control
	in the same spirit as in the discrete-time case \cite[Definition~3]{Waarde2022TAC}.
	\begin{definition}
		The data $\mathfrak{D} = (x_k,u_k)_{k=1}^{N}
		\in \Gamma_{\tau}$ are 
		{\em informative for $H_2$-control with 
			performance $\gamma>0$  under the noise class $\Delta_{\tau,\Theta}$}
		if there exist matrices $\Phi\in \mathbb{S}_{++}^n$,
		$\Psi\in \mathbb{S}_{++}^{n_z}$,
		and $K \in \mathbb{R}^{m \times n}$ such that 
		the inequalities \eqref{eq:H2_LMIs1}--\eqref{eq:H2_LMIs3} 
		hold for all $(A,B) \in \Sigma_{\mathfrak{D},\Theta}$.
	\end{definition}
	
	The combination of the QMI condition in Theorem~\ref{thm:system_cond} 
	and the strict matrix $S$-lemma
	yields
	a characterization of informativity for $H_2$-control
	as in the case of quadratic stabilization.
	\begin{theorem}
		Suppose that the data
		$\mathfrak{D} = (x_k,u_k)_{k=1}^{N}
		\in \Gamma_{\tau}$ and the noise-intensity matrix $\Theta \in \mathbb{S}_+^{2n+m}$
		satisfy Assumption~\ref{assump:Theta}. 
		Assume that $E = 0$. Then
		the following statements are equivalent
		for a fixed $\gamma>0$:
		\begin{enumerate}
			\renewcommand{\labelenumi}{\textup{(\roman{enumi})}}
			\item The data $(x_k,u_k)_{k=1}^{N}$ are 
			informative for $H_2$-control with 
			performance $\gamma$ under the noise class $\Delta_{\tau,\Theta}$.
			\item There exist matrices $\Phi\in \mathbb{S}_{++}^n$,
			$\Psi\in \mathbb{S}_{++}^{n_z}$, and $L \in \mathbb{R}^{m \times n}$, 
			and a scalar $\alpha >0$ such that
			\begin{align}
				&\begin{bmatrix}
					0_{n\times n}
					& -
					\Phi & -L^{\top} \\
					- 	\Phi & 0_{n\times n} & 0_{n\times m} \\
					-L  & 0_{m\times n} & 0_{m\times m}
				\end{bmatrix}- \Omega \Omega^{\top}  + \alpha (ZZ^* - \Theta) \succ 0,
				\label{eq:H2_info_LMIs1}\\
				&\begin{bmatrix}
					\Psi & C\Phi+DL \\
					L^{\top}D^{\top} + \Phi C^{\top} & \Phi
				\end{bmatrix}  \succ 0,
				\label{eq:H2_info_LMIs2} \\
				&\trace\,(\Psi) < \gamma^2,
				\label{eq:H2_info_LMIs3}
			\end{align}
			where the data-associated Gramian
			$ZZ^* \in \mathbb{S}_+^{2n+m}$ is given by
			\eqref{eq:ZZ}.
		\end{enumerate}
		Moreover, if statement (ii) holds, then
		$K \coloneqq L \Phi^{-1}$ is such that 
		$A+BK$ is Hurwitz and $\|G_K\|_{H_2} < \gamma$
		for 
		all systems $(A,B) \in \Sigma_{\mathfrak{D},\Theta}$.
	\end{theorem}
	
	\begin{proof}
		First, we prove the implication (i) $\Rightarrow $ (ii).
		Let $\Phi\in \mathbb{S}_{++}^n$,
		$\Psi \in \mathbb{S}_{++}^{n_z}$, and 
		$K \in \mathbb{R}^{m \times n}$ satisfy
		the inequalities \eqref{eq:H2_LMIs1}--\eqref{eq:H2_LMIs3} for all $(A,B) \in \Sigma_{\mathfrak{D},\Theta}$.
		Define the matrix $\mathcal{M} \in \mathbb{S}^{2n+m}$ by
		\[
		\mathcal{M} \coloneqq
		\begin{bmatrix}
			0_{n\times n}
			& -\Phi & -\Phi K^{\top} \\
			- \Phi & 0_{n\times n} & 0_{n\times m} 
			\\ 
			-K \Phi  & 0_{m\times n} & 0_{m\times m}
		\end{bmatrix} - 
		\Omega \Omega^{\top}.
		\]
		Then
		\[
		\begin{bmatrix}
			I_n \\ A^{\top} \\ B^{\top}
		\end{bmatrix}^{\top}
		\mathcal{M} 
		\begin{bmatrix}
			I_n \\ A^{\top} \\ B^{\top}
		\end{bmatrix} = -
		(A+BK)\Phi-\Phi(A+BK)^{\top} - B_{\Omega}B_{\Omega}^{\top}.
		\]
		By the same argument as in the proof of Theorem~\ref{thm:stabilization},
		the inequality 
		\eqref{eq:H2_LMIs1} implies that
		$\mathcal{M} - \alpha \mathcal{N} \succ 0$
		for some scalar $\alpha >0$, 
		where $\mathcal{N}\in \mathbb{S}^{2n+m}$ is as in \eqref{eq:N_def}.
		Setting $L \coloneqq K\Phi$, we obtain 
		the first inequality \eqref{eq:H2_info_LMIs1}.
		Since 
		the inequalities \eqref{eq:H2_LMIs2}
		and \eqref{eq:H2_info_LMIs2} are equivalent
		by the Schur complement argument,
		statement (ii) follows.
		
		Next, we prove
		the implication (ii) $\Rightarrow $ (i).
		Let $\Phi\in \mathbb{S}_{++}^n$,
		$\Psi\in \mathbb{S}_{++}^{n_z}$, $L \in \mathbb{R}^{m \times n}$,
		and $\alpha >0$ satisfy the inequalities
		\eqref{eq:H2_info_LMIs1}--\eqref{eq:H2_info_LMIs3}.
		Set $K \coloneqq L\Phi^{-1}$.
		By \eqref{eq:H2_info_LMIs1},
		one can derive the first inequality \eqref{eq:H2_LMIs1}
		in the same way as in the proof of
		Theorem~\ref{thm:stabilization}, and
		the Schur complement argument shows that
		the inequalities \eqref{eq:H2_LMIs2}
		and \eqref{eq:H2_info_LMIs2} are equivalent.
		Thus, $(x_k,u_k)_{k=1}^{N}$ are 
		informative for $H_2$-control with 
		performance $\gamma$  under the noise class $\Delta_{\tau,\Theta}$. Moreover,
		$A+BK$ is Hurwitz and $\|G_K\|_{H_2} < \gamma$
		for 
		all $(A,B) \in \Sigma_{\mathfrak{D},\Theta}$.
	\end{proof}
	
	\begin{remark}[Weighting matrix for 
		external signals]
		\label{rem:H2_Sigma}
		If
		$\widetilde \Omega \in \mathbb{R}^{(2n+m) \times \widetilde n_w}$
		satisfies
		$\widetilde \Omega\widetilde \Omega^{\top} \preceq
		\Omega\Omega^{\top}$, then
		the LMI \eqref{eq:H2_info_LMIs1} 
		remains valid when $\Omega $ is replaced by
		$\widetilde \Omega$.
		This implies that
		the weighting matrix $\Omega$ may be chosen 
		based on upper bound for the external signals.	\hspace*{\fill} $\triangle$ 
	\end{remark}
	
	\subsubsection{$H_{\infty}$-control}
	\label{sec:Hinf_control}
	We first recall that
	the following statements are equivalent
	for fixed  
	$\gamma >0$
	and $K\in \mathbb{R}^{m \times n}$
	(see, e.g., \cite[Section~5.2]{Duan2013}):
	\begin{enumerate}
		\renewcommand{\labelenumi}{\textup{(\roman{enumi})}}
		\item 
		$A+BK$ is Hurwitz and 
		$\|G_K\|_{H_{\infty}} < \gamma$.
		\item There exists a matrix $\Phi \in \mathbb{S}_{++}^n$ such that
		\begin{equation}
			\label{eq:Hinf_LMIs}
			\begin{bmatrix}
				(A+BK)\Phi+ \Phi(A+BK)^{\top} + B_{\Omega}B_{\Omega}^{\top} & 
				\Phi(C+DK)^{\top}  + B_{\Omega}E^{\top} \\
				(C+DK)\Phi + EB_{\Omega}^{\top} & EE^{\top} - \gamma^2 I_{n_z}
			\end{bmatrix} \prec 0.
		\end{equation}
	\end{enumerate}
	
	Based on this equivalence, we define
	the notion of informativity for $H_{\infty}$-control
	as in the discrete-time case
	\cite[Definition~19]{Waarde2022TAC}.
	\begin{definition}
		The data $\mathfrak{D} = (x_k,u_k)_{k=1}^{N}
		\in \Gamma_{\tau}$ are 
		{\em informative for $H_{\infty}$-control with performance $\gamma>0$  under the noise class $\Delta_{\tau,\Theta}$}
		if there exist matrices $\Phi \in \mathbb{S}_{++}^n$ and 
		$K \in \mathbb{R}^{m \times n}$ 
		such that
		the LMI \eqref{eq:Hinf_LMIs} holds for all $(A,B) \in \Sigma_{\mathfrak{D},\Theta}$.
	\end{definition}
	
	This informativity property is characterized using an argument analogous to that for 
	quadratic stabilization and $H_2$-control.
	\begin{theorem}
		\label{thm:H_inf}
		Suppose that the data
		$\mathfrak{D} = (x_k,u_k)_{k=1}^{N}
		\in \Gamma_{\tau}$ and the noise-intensity matrix $\Theta \in \mathbb{S}_+^{2n+m}$
		satisfy Assumption~\ref{assump:Theta}. 
		Then
		the following statements are equivalent
		for a fixed $\gamma >0$:
		\begin{enumerate}
			\renewcommand{\labelenumi}{\textup{(\roman{enumi})}}
			\item The data $(x_k,u_k)_{k=1}^{N}$ are 
			informative for $H_\infty$-control with 
			performance $\gamma$  under the noise class $\Delta_{\tau,\Theta}$.
			\item There exist matrices $\Phi \in \mathbb{S}_{++}^n$ and 
			$L \in \mathbb{R}^{m \times n}$, 
			and a scalar $\alpha >0$ 
			such that
			\begin{equation}
				\label{eq:Hinf_info_LMIs}
				\begin{bmatrix}
					0_{n\times n}
					& - \Phi & - L^{\top} &
					\Phi C^{\top} +L^{\top} D^{\top} \\
					-\Phi & 0_{n\times n} & 0_{n\times m} & 0_{n\times n_z} \\
					-L & 0_{m\times n} & 0_{m\times m} & 0_{m\times n_z} \\
					C \Phi  +DL   & 0_{n_z\times n} & 0_{n_z\times m} & 
					\gamma^2 I_{n_z} 
				\end{bmatrix} -
				\begin{bmatrix}
					\Omega \\ E
				\end{bmatrix}
				\begin{bmatrix}
					\Omega \\ E
				\end{bmatrix}^{\top}
				+ 
				\begin{bmatrix}
					\alpha(ZZ^*-\Theta) & 0_{(2n+m)\times n_z} \\
					0_{n_z \times (2n+m)} & 0_{n_z\times n_z}
				\end{bmatrix}\succ 0,
			\end{equation}
			where the data-associated Gramian
			$ZZ^* \in \mathbb{S}_+^{2n+m}$ is given by
			\eqref{eq:ZZ}.
		\end{enumerate}
		Moreover, if statement (ii) holds, then
		$K \coloneqq L \Phi^{-1}$ is such that 
		$A+BK$ is Hurwitz and $\|G_K\|_{H_{\infty}} < \gamma$
		for 
		all systems $(A,B) \in \Sigma_{\mathfrak{D},\Theta}$.
	\end{theorem}

	\begin{proof}
		First, we prove the implication (i) $\Rightarrow $ (ii).
		Let $\Phi \in \mathbb{S}_{++}^n$ and 
		$K \in \mathbb{R}^{m \times n}$ 
		satisfy
		the inequality \eqref{eq:Hinf_LMIs} for all $(A,B) \in \Sigma_{\mathfrak{D},\Theta}$.
		Define $\Lambda_{\gamma}  \in \mathbb{S}^{n_z}$
		and $C_K \in \mathbb{R}^{n_z \times n}$ by
		\begin{equation}
			\label{eq:Lambda_CK_def}
			\Lambda_{\gamma}  \coloneqq\gamma^2 I_{n_z} -  EE^{\top}
			\quad \text{and} \quad 
			C_K \coloneqq C+DK.
		\end{equation}
		By \eqref{eq:Hinf_LMIs}, we obtain
		$\Lambda_{\gamma}  \succ 0$. Therefore,
		the Schur complement argument shows that
		the LMI~\eqref{eq:Hinf_LMIs} holds if and only if
		\begin{equation}
			\label{eq:Hinf_LMIs_SC}
			-(A+BK)\Phi-\Phi(A+BK)^{\top} - B_{\Omega}B_{\Omega}^{\top}
			- (\Phi C_K^{\top} + B_{\Omega}E^{\top}) \Lambda_{\gamma} ^{-1} 
			(C_K \Phi + E B_{\Omega}^{\top}) \succ 0.
		\end{equation}
		
		Define 
		$J_{n,m} \in \mathbb{R}^{(2n+m) \times n}$ by
		\[
		J_{n,m} \coloneqq 
		\begin{bmatrix}
			I_n \\ 0_{n\times n} \\ 0_{m\times n}
		\end{bmatrix}.
		\]
		Then
		\begin{equation}
			\label{eq:Lambda_inv_QMI}
			(\Phi C_K^{\top} + B_{\Omega}E^{\top}) \Lambda_{\gamma} ^{-1} 
			(C_K \Phi + E B_{\Omega}^{\top})  = 
			\begin{bmatrix}
				I_n \\ A^{\top} \\ B^{\top}
			\end{bmatrix}^{\top}
			\left(
			J_{n,m}\Phi C_K^{\top}
			+ \Omega E^{\top}
			\right) \Lambda_{\gamma} ^{-1}
			\left(
			C_K \Phi 
			J_{n,m}^{\top}
			+ E \Omega^{\top} 
			\right)
			\begin{bmatrix}
				I_n \\ A^{\top} \\ B^{\top}
			\end{bmatrix}.
		\end{equation}
		Define the matrices $\mathcal{M}_0,
		\mathcal{M} \in \mathbb{S}^{2n+m}$ by
		\begin{align}
			\label{eq:M0_def_Hinf}
			\mathcal{M}_0 &\coloneqq 
			\begin{bmatrix}
				0_{n\times n}  & -
				\Phi & - \Phi K^{\top} \\
				-\Phi & 0_{n\times n} & 0_{n\times m} \\
				-K\Phi  & 0_{m\times n} & 0_{m\times m} 
			\end{bmatrix}, \\
			\mathcal{M} &\coloneqq
			\mathcal{M}_0 - 
			\Omega \Omega^{\top} - 
			\left(
			J_{n,m}\Phi C_K^{\top}
			+ \Omega E^{\top}
			\right) \Lambda_{\gamma} ^{-1}
			\left(
			C_K \Phi 
			J_{n,m}^{\top}
			+ E \Omega^{\top} 
			\right).
			\label{eq:M_def_Hinf}
		\end{align}
		Using \eqref{eq:Lambda_inv_QMI},
		we can write \eqref{eq:Hinf_LMIs_SC} as
		\[
		\begin{bmatrix}
			I_n \\ A^{\top} \\ B^{\top}
		\end{bmatrix}^{\top}
		\mathcal{M}
		\begin{bmatrix}
			I_n \\ A^{\top} \\ B^{\top}
		\end{bmatrix} \succ 0.
		\]
		Therefore, we obtain
		$\mathcal{M} - \alpha \mathcal{N} \succ 0$
		for some $\alpha > 0$ as in the proof of Theorem~\ref{thm:stabilization},
		where $\mathcal{N}\in \mathbb{S}^{2n+m}$ is defined by \eqref{eq:N_def}.
		The Schur complement argument yields
		\begin{equation}
			\label{eq:M0_Omega}
			\mathcal{M} - \alpha \mathcal{N} \succ 0
			\quad \Leftrightarrow \quad 
			\begin{bmatrix}
				\mathcal{M}_0 - \Omega \Omega^{\top} - \alpha \mathcal{N} & 
				J_{n,m}\Phi C_K^{\top}
				+ \Omega E^{\top} \\
				C_K \Phi J_{n,m}^{\top} + 
				E \Omega^{\top}  & \Lambda_{\gamma} 
			\end{bmatrix} \succ 0.
		\end{equation}
		Since
		\begin{equation}
			\label{eq:M0_Omega_trans}
			\begin{bmatrix}
				\mathcal{M}_0 - \Omega \Omega^{\top}& 
				\Omega E^{\top} \\
				E\Omega^{\top}   & \Lambda_{\gamma} 
			\end{bmatrix} =
			\begin{bmatrix}
				\mathcal{M}_0 & 0_{(2n+m)\times n_z} \\
				0_{n_z\times (2n+m)} & \gamma^2 I_{n_z} 
			\end{bmatrix}
			-
			\begin{bmatrix}
				\Omega \\ E
			\end{bmatrix} 
			\begin{bmatrix}
				\Omega \\ E
			\end{bmatrix}^{\top},
		\end{equation}
		we conclude that the
		LMI \eqref{eq:Hinf_info_LMIs} holds,
		by setting
		$L \coloneqq K\Phi$ in the definition \eqref{eq:M0_def_Hinf} 
		of $\mathcal{M}_0$.
		
		Next, we prove
		the implication (ii) $\Rightarrow $ (i).
		Let $\Phi \in \mathbb{S}_{++}^n$,
		$L \in \mathbb{R}^{m \times n}$, 
		and  $\alpha >0$ satisfy the LMI~\eqref{eq:Hinf_info_LMIs}.
		Set $K \coloneqq L\Phi^{-1}$. Let $\Lambda_{\gamma}  \in \mathbb{S}^{n_z}$
		and $C_K \in \mathbb{R}^{n_z \times n}$ be as in \eqref{eq:Lambda_CK_def}. Then 
		the LMI \eqref{eq:Hinf_info_LMIs} implies that 
		$\Lambda_{\gamma} \succ 0$; see \eqref{eq:M0_Omega_trans}. Hence, using the inverse $\Lambda_{\gamma} ^{-1}$, we can define 
		$\mathcal{M} \in \mathbb{S}^{2n+m}$ by
		\eqref{eq:M_def_Hinf}.
		Let  $\mathcal{N}\in \mathbb{S}^{2n+m}$ be as in \eqref{eq:N_def}.
		Since \eqref{eq:M0_Omega} remains valid without 
		assuming that statement~(i) holds, 
		we obtain $\mathcal{M}-\alpha \mathcal{N}\succ 0$
		by the LMI \eqref{eq:Hinf_info_LMIs}.
		The rest of the proof is similar to that of
		Theorem~\ref{thm:stabilization}, applying the equivalence of
		\eqref{eq:Hinf_LMIs} and \eqref{eq:Hinf_LMIs_SC}.
	\end{proof}
	
	\begin{remark}[Weighting matrix for 
		external signals]
		\label{rem:Hinf_Sigma}
		Suppose that 
		$\widetilde \Omega \in \mathbb{R}^{(2n+m) \times \widetilde n_w}$
		and $\widetilde E \in \mathbb{R}^{n_z \times \widetilde n_w}$
		satisfy
		\[
		\begin{bmatrix}
			\widetilde \Omega \\ \widetilde E
		\end{bmatrix}
		\begin{bmatrix}
			\widetilde \Omega \\ \widetilde E
		\end{bmatrix}^{\top} \preceq 
		\begin{bmatrix}
			\Omega \\ E
		\end{bmatrix}
		\begin{bmatrix}
			\Omega \\ E
		\end{bmatrix}^{\top}.
		\]
		Then
		the LMI \eqref{eq:Hinf_info_LMIs}
		is still feasible when $\Omega $ and 
		$E$ are
		replaced by
		$\widetilde \Omega$ and $\widetilde E$, respectively.
		In this case, the matrices 
		$\Omega $ and 
		$E$
		can be designed using
		upper bounds for external signals, as
		in the $H_2$-control case (see Remark~\ref{rem:H2_Sigma}).	\hspace*{\fill} $\triangle$ 
	\end{remark}
	
	\section{Error analysis in signal reconstruction from sampled data}
	\label{sec:error_analysis}
	In this section, we reconstruct continuous-time signals 
	from sampled data and analyze the reconstruction error using synthesis operators.
	This error can be interpreted as measurement 
	noise within the framework of
	continuous-time data-driven control
	discussed in Sections~\ref{sec:data_consistent_systems}
	and \ref{sec:Data_driven_control}. Consequently,
	the following error analysis can be used to determine
	the noise-intensity matrix $\Theta$
	in Assumption~\ref{assump:Theta}.
	In  Section~\ref{sec:reconstruction_noise_free_case},
	we consider the case where
	noise-free sampled data are available, and 
	give
	upper bounds for the Gramians
	associated with the reconstruction error.
	In Section~\ref{sec:reconstruction_noisy_case},
	this error analysis is extended 
	to noisy sampled data.

	Throughout this section, we fix
	\[
	0 = t_0 < t_1 < \dots < t_N = \tau
	\]
	and $x \in \Hso([0,\tau];\mathbb{R}^n)$.
	Let 
	\begin{equation}
		\label{eq:h_def}
		h \coloneqq 
		\max\{t_{\ell+1}-t_{\ell}:\ell=0,\dots,N-1 \}.
	\end{equation}
	We make the following assumption on
	the $\Ls$-norm of the derivative of $x$.
	\begin{assumption}
		\label{assump:deri_bound}
		There exist constants $M\in [0,\pi/h)$ and $c \geq 0$ 
		such that  
		\begin{equation}
			\label{eq:x_deriv_bound}
			\|x'\|_{\Ls} \leq M \|x\|_{\Ls}+c.
		\end{equation}
	\end{assumption}
	The constants $M$ and $c$ in this assumption can be computed when
	upper bounds for the unknown system matrices and the noise are available 
	in the contexts of Sections~\ref{sec:data_consistent_systems}
	and \ref{sec:Data_driven_control}.

	\subsection{Noise-free sampled data}
	\label{sec:reconstruction_noise_free_case}
	Based on the sampled data $(x(t_{\ell}))_{\ell = 0}^N$, we define
	$\widehat x \in \Hso([0,\tau];\mathbb{R}^n)$ as the piecewise linear interpolant given by
	\begin{equation}
		\label{eq:hat_x_def}
		\widehat x(t) \coloneqq \frac{t_{\ell+1} - t}{t_{\ell+1} - t_{\ell}}x(t_{\ell}) + 
		\frac{t - t_{\ell}}{t_{\ell+1} - t_{\ell}}x(t_{\ell+1}) 
		,\quad t_{\ell} \leq t < t_{\ell+1},\,\ell=0,\dots,N-1.
	\end{equation}
	We consider the situation where 
	the original signal $x$ is unknown but
	the reconstructed signal $\widehat x$,  
	the maximum sampling interval $h$, and 
	the constants $M,c$ in Assumption~\ref{assump:deri_bound} 
	are available.
	The following proposition shows that 
	the Gramians associated with reconstruction errors are bounded in terms of $h^2$
	and $\widehat x$.
	\begin{proposition}
		\label{prop:PL_app}
		Let $h>0$ be as in \eqref{eq:h_def}, and suppose that $x \in \Hso([0,\tau];\mathbb{R}^n)$ satisfies
		Assumption~\ref{assump:deri_bound}.
		Define $\widehat x \in \Hso([0,\tau];\mathbb{R}^n)$
		by
		\eqref{eq:hat_x_def}. 
		Let $P$ and $Q$ 
		be the synthesis operators associated with $x' - \widehat x'$ and $x - \widehat x$,
		respectively. Then
		\begin{align}
			P P^*  &\preceq h^2
			\left( \frac{
				(M
				\|\widehat x\|_{\Ls} +  c)^2}{(\pi - hM)^2}
			-
			\frac{\|\widehat x'\|_{\Ls}^2}{\pi^2} \right) I_n,
			\label{eq:PPad_bound} \\
			Q Q^*  &\preceq  \frac{ \tau^2 h^2}{\pi^2} 
			\left( \frac{
				(M
				\|\widehat x\|_{\Ls} +  c)^2}{(\pi - hM)^2}
			-
			\frac{\|\widehat x'\|_{\Ls}^2}{\pi^2} \right)I_n.
			\label{eq:QQad_bound} 
		\end{align}
	\end{proposition}
	
	Before proceeding to the proof of Proposition~\ref{prop:PL_app}, we present a remark on bounding a block matrix by a block-diagonal matrix.
	Applying the inequalities \eqref{eq:PPad_bound}
	and \eqref{eq:QQad_bound} to 
	this block-diagonal bound yields 
	a diagonal bound for 
	the noise-associated Gramian $WW^*$ given by \eqref{eq:WW}.
	\begin{remark}[Block-diagonal bound]
		\label{rem:young}
		Let $\varepsilon >0$ and 
		$\mu,\nu \in \mathbb{R}^n$ be arbitrary.
		Let $P, Q \in \mathcal{L}(Y,\mathbb{R}^n)$
		for some real Hilbert space $Y$.
		If we define $\phi \in Y$ by
		\[
		\phi \coloneqq 
		\sqrt{\varepsilon} P^*\mu  - \frac{1}{\sqrt{\varepsilon} }
		Q^*\nu,
		\]
		then 
		\[
		0 \leq \|\phi\|_{Y}^2 = 
		\varepsilon  \|P^* \mu\|_{Y}^2 + 
		\frac{\|Q^* \nu\|_{Y}^2}{\varepsilon} - 
		2 \langle P^*\mu, Q^*\nu \rangle_{Y}.
		\]
		Therefore, 
		\[
		2 \langle P^*\mu, Q^*\nu \rangle_{Y} \leq \varepsilon \|P^* \mu\|_{Y}^2 + 
		\frac{\|Q^* \nu\|_{Y}^2}{\varepsilon},
		\]
		which yields
		\[
		\begin{bmatrix}
			PP^* & -PQ^* \\ -QP^* & QQ^*
		\end{bmatrix}\preceq
		\begin{bmatrix}
			(1+\varepsilon)PP^* &  0_{n\times n} \\ 0_{n\times n} & (1+1/\varepsilon)QQ^*
		\end{bmatrix}.
		\]
		Extending this argument, we also obtain
		\[
		\begin{bmatrix}
			PP^* & -PQ^* & -PR^* \\ -QP^* & QQ^* & 
			QR^* \\
			-RP^* & RQ^* & RR^*
		\end{bmatrix}\preceq
		\begin{bmatrix}
			(1+\varepsilon_1 + \varepsilon_2)PP^* &  0_{n\times n} &0_{n\times m} \\ 
			0_{n\times n} & (1+1/\varepsilon_1 + \varepsilon_3)QQ^* & 0_{n\times m} \\
			0_{m\times n} & 0_{m\times n} & (1+1/\varepsilon_2 + 1/\varepsilon_3)RR^*
		\end{bmatrix}
		\]
		for all $\varepsilon_1,\varepsilon_2,\varepsilon_3 >0$, where $R \in \mathcal{L}(Y,\mathbb{R}^m)$.
		\hspace*{\fill} $\triangle$ 
	\end{remark}

	We need two auxiliary results
	for the proof of Proposition~\ref{prop:PL_app}.
	The following lemma provides 
	upper bounds for the Gramians
	in terms of
	the associated $\Ls$-function.
	Note that the operator $P$ considered in Proposition~\ref{prop:PL_app} 
	coincides with
	the negative of the differentiated synthesis operator associated with
	$x-\widehat x$.
	\begin{lemma}
		\label{lem:synthesis_ad_bound}
		Let $F$ and $F_{\sd}$ be 
		the synthesis operator and 
		the differentiated synthesis operator associated 
		with $f \in \Ls([0,\tau];\mathbb{R}^n)$. Then
		\begin{align}
			F_{\sd}F_{\sd }^* &\preceq
			\|f\|_{\Ls}^2 I_n , \label{eq:TdTd_bound}\\
			FF^* &\preceq
			\frac{\tau^2}{\pi^2} \|f\|_{\Ls}^2 I_n. 
			\label{eq:TT_bound}
		\end{align}
	\end{lemma}
	\begin{proof}
		By \cite[Lemma~2.3.b)]{Wakaiki2025Cont},
		\[
		F_{\sd} F_{\sd}^* =
		\int_0^{\tau} f(t)f(t)^{\top} dt
		-
		\frac{1}{\tau}
		\int_0^{\tau} f(t)dt
		\int_0^{\tau} f(t)^{\top}dt.
		\]
		Let $\mu \in \mathbb{R}^n$ be arbitrary.
		We have
		\[
		\mu^{\top} F_{\sd} F_{\sd}^* \mu =
		\int_0^{\tau} |f(t)^{\top} \mu|^2 dt- 
		\frac{1}{\tau}
		\left|
		\int_0^{\tau} f(t)^{\top} \mu dt
		\right|^2 \leq \int_0^{\tau} |f(t)^{\top} \mu|^2 dt.
		\]
		Since
		\begin{equation}
			\label{eq:fmu_L2}
			\int_0^{\tau} |f(t)^{\top} \mu|^2 dt \leq 
			\int_0^{\tau} 
			\|f(t)\|^2 \|\mu\|^2 dt \leq \|f\|_{\Ls}^2\|\mu\|^2,
		\end{equation}
		we obtain \eqref{eq:TdTd_bound}.
		
		Define
		the orthonormal basis $(\psi_j)_{j \in \mathbb{N}}$ of $\Ls[0,\tau]$ by
		\[
		\psi_j(t) \coloneqq \sqrt{\frac{2}{\tau}} \sin\left(
		\frac{j \pi }{\tau} t
		\right),\quad t \in [0,\tau].
		\]
		As seen in the proof of  \cite[Proposition~A]{Wakaiki2025Cont},
		we have
		\begin{equation}
			\label{eq:T_ad_series}
			(F^*\mu)(t) 
			= \frac{\tau^2}{\pi^2} \sum_{j=1}^{\infty}
			\frac{\langle f(\cdot)^{\top}\mu, \psi_j \rangle_{\Ls} }{j^2} \psi_j(t)
		\end{equation}
		for all $t \in [0,\tau]$.
		For each $j \in \mathbb{N}$,
		\begin{equation}
			\label{eq:T_psi}
			\mu^{\top} F\psi_j = 
			\mu^{\top} 
			\int_0^{\tau} \psi_j(t)f(t)dt 
			=\langle f(\cdot)^{\top}\mu,\psi_j \rangle_{\Ls}.
		\end{equation}
		By \eqref{eq:T_ad_series} and \eqref{eq:T_psi},
		\begin{equation}
			\label{eq:muTTmu}
			\mu^{\top} FF^*\mu =
			\frac{\tau^2}{\pi^2} \sum_{j=1}^{\infty}
			\frac{\langle f(\cdot)^{\top}\mu, \psi_j \rangle_{\Ls}^2 }{j^2} \leq 
			\frac{\tau^2}{\pi^2} \|f(\cdot)^{\top} \mu\|_{\Ls}^2.
		\end{equation}
		Combining \eqref{eq:fmu_L2} and \eqref{eq:muTTmu},
		we conclude that \eqref{eq:TT_bound} holds.
	\end{proof}
	
	The next result gives an upper bound for
	the reconstruction error $x - \widehat x$.
	\begin{lemma}
		\label{lem:x_hatx_diff}
		Let $h>0$ be as in \eqref{eq:h_def}, and 
		define $\widehat x \in \Hso([0,\tau];\mathbb{R}^n)$
		by
		\eqref{eq:hat_x_def}. 
		Then
		\begin{equation}
			\label{eq:x_hat_x_L2}
			\|x - \widehat x\|_{\Ls}^2 \leq 
			\frac{h^2}{\pi^2} (\|x'\|_{\Ls}^2 - \|\widehat x'\|_{\Ls}^2).
		\end{equation}
	\end{lemma}
	\begin{proof}
		Fix $i\in \{1,\dots,n\}$.
		For $t \in [0,\tau]$,
		let $f(t)$ and $\widehat f(t)$ be the $i$-th
		elements of $x(t)$ and $\widehat x(t)$, respectively.
		It suffices to show that 
		\begin{equation}
			\label{eq:xi_hatxi_diff}
			\|f - \widehat f\|_{\Ls}^2 \leq 
			\frac{h^2}{\pi^2} (\|f'\|_{\Ls}^2 - \|\widehat f'\|_{\Ls}^2).
		\end{equation}

		Let $\ell \in \{ 0,\dots, N-1\}$ be arbitrary.
		Since $f - \widehat f \in \Hs[t_{\ell},t_{\ell+1}]$, the  Wirtinger inequality (see, e.g., 
		\cite[Sec.~1.7.2]{Dym1972}) yields
		\begin{equation}
			\label{eq:subinterval_bound}
			\int_{t_{\ell}}^{t_{\ell+1}} |f(t) - \widehat f(t)|^2dt \leq \frac{(t_{\ell+1}-t_{\ell})^2}{\pi^2} 
			\int_{t_{\ell}}^{t_{\ell+1}} |f'(t) - \widehat f'(t)|^2dt .
		\end{equation}
		We have
		\begin{equation}
			\label{eq:f'_hatf'_subinterval}
			\int_{t_{\ell}}^{t_{\ell+1}} |f'(t) - \widehat f'(t)|^2dt=
			\int_{t_{\ell}}^{t_{\ell+1}} |f'(t)|^2dt - 
			2 \int_{t_{\ell}}^{t_{\ell+1}} f'(t) \widehat f'(t)dt + 
			\int_{t_{\ell}}^{t_{\ell+1}} |\widehat f'(t)|^2dt.
		\end{equation}
		Since $\widehat f'$ is constant on $[t_{\ell},t_{\ell+1})$, 
		the cross term satisfies
		\[
		\int_{t_{\ell}}^{t_{\ell+1}} f'(t) \widehat f'(t)dt = 
		\frac{f(t_{\ell+1}) - f(t_{\ell})}{t_{\ell+1}-t_{\ell}}\int_{t_{\ell}}^{t_{\ell+1}} f'(t)dt =
		\frac{|f(t_{\ell+1}) - f(t_{\ell})|^2}{t_{\ell+1} - t_{\ell}} = 
		\int_{t_{\ell}}^{t_{\ell+1}} |\widehat f'(t)|^2dt.
		\]
		Substituting this into \eqref{eq:f'_hatf'_subinterval}, we obtain
		\begin{equation}
			\label{eq:subinterval_eq}
			\int_{t_{\ell}}^{t_{\ell+1}} |f'(t) - \widehat f'(t)|^2dt =
			\int_{t_{\ell}}^{t_{\ell+1}} |f'(t)|^2dt -
			\int_{t_{\ell}}^{t_{\ell+1}} |\widehat f'(t)|^2dt .
		\end{equation}
		By \eqref{eq:subinterval_bound} and \eqref{eq:subinterval_eq},
		\[
		\int_{t_{\ell}}^{t_{\ell+1}} |f(t) - \widehat f(t)|^2dt  \leq \frac{(t_{\ell+1}-t_{\ell})^2}{\pi^2} 
		\left(\int_{t_{\ell}}^{t_{\ell+1}} |f'(t)|^2dt -
		\int_{t_{\ell}}^{t_{\ell+1}} |\widehat f'(t)|^2dt \right).
		\]
		Since $\ell \in \{ 0,\dots, N-1\}$ is arbitrary, it follows that 
		\begin{align*}
			\int_{0}^{\tau}
			|f(t) - \widehat f(t)|^2dt &=
			\sum_{\ell =0}^{N-1} \int_{t_{\ell}}^{t_{\ell+1}} |f(t) - \widehat f(t)|^2dt  \\
			&\leq 
			\frac{h^2}{\pi^2} 
			\sum_{\ell =0}^{N-1} \left(\int_{t_{\ell}}^{t_{\ell+1}} |f'(t)|^2dt -
			\int_{t_{\ell}}^{t_{\ell+1}} |\widehat f'(t)|^2dt \right) \\
			&= \frac{h^2}{\pi^2} 
			\left(\int_{0}^{\tau}
			|f'(t)|^2dt -\int_{0}^{\tau}
			|\widehat f'(t)|^2dt 
			\right).
		\end{align*}
		Thus, the assertion \eqref{eq:xi_hatxi_diff} holds.
	\end{proof}
	
	After these preparations,
	we are in a position to prove Proposition~\ref{prop:PL_app}.
	\begin{proof}[Proof of Proposition~\ref{prop:PL_app}]
		Lemma~\ref{lem:x_hatx_diff} shows that
		\[
		\|x - \widehat x\|_{\Ls} 
		\leq \frac{h}{\pi}\|x'\|_{\Ls}.
		\]
		Combining this with the assumption
		\eqref{eq:x_deriv_bound}, we obtain
		\begin{align*}
			\|\widehat x\|_{\Ls} 
			&\geq \|x\|_{\Ls} - 
			\|x - \widehat x\|_{\Ls} \\
			&\geq \|x\|_{\Ls} - 
			\frac{h}{\pi}
			\|x'\|_{\Ls}  \\
			& \geq 
			\left(
			1 - \frac{hM}{\pi}
			\right) \|x\|_{\Ls} -  \frac{h}{\pi}c.
		\end{align*}
		Since $hM< \pi$, it follows that
		\[
		\|x\|_{\Ls} \leq \frac{\pi}{\pi - hM} \|\widehat x\|_{\Ls} +
		\frac{h}{\pi - hM}c.
		\]
		Using again the assumption~\eqref{eq:x_deriv_bound}, we obtain
		\begin{equation}
			\label{eq:hatx_bound}
			\|x'\|_{\Ls} \leq M \left(
			\frac{\pi }{\pi - hM} \|\widehat x\|_{\Ls} +
			\frac{h}{\pi - hM}c
			\right) + c =
			\frac{\pi }{\pi - hM} (M \|\widehat x\|_{\Ls} + c).
		\end{equation}
		By Lemma~\ref{lem:x_hatx_diff} and 
		\eqref{eq:hatx_bound}, 
		\begin{equation}
			\label{eq:x_hatx_diff_bound}
			\|x-\widehat x\|_{\Ls}^2 \leq \frac{h^2}{\pi^2}\left( 
			\|x'\|_{\Ls}^2 - \|\widehat x'\|_{\Ls}^2\right)
			\leq h^2 \left( 
			\frac{
				(M
				\|\widehat x\|_{\Ls} +  c)^2}{(\pi - hM)^2}
			-
			\frac{\|\widehat x'\|_{\Ls}^2}{\pi^2}\right).
		\end{equation}
		From Lemma~\ref{lem:synthesis_ad_bound} and
		\eqref{eq:x_hatx_diff_bound}, we conclude that 
		the assertions 
		\eqref{eq:PPad_bound} 
		and \eqref{eq:QQad_bound}  hold.
	\end{proof}
	
	\subsection{Noisy sampled data}
	\label{sec:reconstruction_noisy_case}
	We now consider the scenario where
	the sampled data are corrupted by measurement noise.
	We denote by
	$\delta_{\ell,i}$ the discrete noise added to
	the $i$-th element of the sampled data $x(t_{\ell})$.
	Let $\delta_{\ell,i}$ satisfy 
	$|\delta_{\ell,i} | \leq \widebar \delta_i$
	for all $\ell=0,\dots,N$ and $i=1,\dots,n$, where 
	$\widebar \delta_i >0$ is a known constant. 
	We set 
	\[\delta_{\ell} \coloneqq 
	\begin{bmatrix}
		\delta_{\ell,1} & \cdots & 
		\delta_{\ell,n}
	\end{bmatrix}^{\top}\quad \text{and} \quad 
	\widebar \delta \coloneqq 
	\begin{bmatrix}
		\widebar \delta_1 & \cdots & 
		\widebar \delta_n
	\end{bmatrix}^{\top}.
	\]
	Using
	the noisy sampled data $(x(t_{\ell}) + \delta_{\ell})_{\ell=0}^N$,
	we define the piecewise linear function $\widebar x \in \Hso([0,\tau];\mathbb{R}^n)$ by
	\begin{equation}
		\label{eq:tilde_x_def}
		\widebar x(t) \coloneqq \frac{t_{\ell+1} - t}{t_{\ell+1} - t_{\ell}}(x(t_{\ell})+\delta_{\ell}) + 
		\frac{t - t_{\ell}}{t_{\ell+1} - t_{\ell}}(x(t_{\ell+1})+\delta_{\ell+1}) 
		,\quad t_{\ell} \leq t < t_{\ell+1},\,\ell=0,\dots,N-1.
	\end{equation}
	Instead of the noise-free signal $\widehat{x}$ defined by \eqref{eq:hat_x_def}, we assume that only the noise bound $\widebar{\delta}$ and the signal $\widebar{x}$ reconstructed from the noisy data are available.
	
	We start by investigating the error $\widehat x - \widebar x$
	due to the measurement noise.
	Recall that the following 
	inequality holds between arithmetic and
	geometric means:
	\begin{equation}
		\label{eq:Ineq_AGM}
		\sqrt{ab} \leq \frac{a+b}{2},\quad a,b \geq 0,
	\end{equation}
	where equality holds if and only if $a=b$.
	
	\begin{lemma}
		\label{lem:til_hat_x_diff}
		Let $h>0$ be as in \eqref{eq:h_def}, and 
		define $\widehat x,\widebar x \in \Hso([0,\tau];\mathbb{R}^n)$
		by
		\eqref{eq:hat_x_def} and \eqref{eq:tilde_x_def},
		respectively.
		Then 
		\begin{equation}
			\label{eq:til_hat_x_diff}
			\|\widehat x - \widebar x\|_{\Ls} \leq \sqrt{\tau} \|\widebar{\delta}\|_{\mathbb{R}^n}.
		\end{equation}
	\end{lemma}
	\begin{proof}
		Fix $i \in \{ 1,\dots, n\}$.
		Let 
		$\widehat x_i(t)$ and $\widebar x_i(t)$ be
		the $i$-th element of $\widehat x(t)$ and 
		$\widebar x(t)$ for $t \in [0,\tau]$, respectively.
		It suffices to show that
		\begin{equation}
			\label{eq:hxi_txi_error}
			\int_0^{\tau} |\widehat x_i(t) - \widebar x_i(t)|^2dt 
			\leq
			\tau \widebar{\delta}_i^2.
		\end{equation}
		Observe first that 
		\begin{equation}
			\label{eq:integral_decomp}
			\int_0^{\tau} |\widehat x_i(t) - \widebar x_i(t)|^2dt 
			=
			\sum_{\ell=0}^{N-1}
			\int_{t_{\ell}}^{t_{\ell+1}} 
			|\widehat x_i(t) - \widebar x_i(t)|^2dt.
		\end{equation}
		Let $\ell \in \{ 0,\dots,N-1\}$.
		Since
		\[
		\widebar x_i(t) - \widehat x_i(t) =
		\frac{\delta_{\ell+1,i} - \delta_{\ell,i}}{t_{\ell+1} - t_{\ell}}
		t + 
		\frac{t_{\ell+1}\delta_{\ell,i} - t_{\ell}\delta_{\ell+1,i}}{t_{\ell+1} - t_{\ell}}
		\]
		for all $t \in [t_{\ell},t_{\ell+1})$,
		a routine calculation shows that
		\begin{equation}
			\label{eq:int_hxi_txi}
			\int_{t_{\ell}}^{t_{\ell+1}} 
			|\widehat x_i(t) - \widebar x_i(t)|^2dt =
			\frac{t_{\ell+1} - t_{\ell}}{3} (\delta_{\ell,i}^2 + 
			\delta_{\ell,i}\delta_{\ell+1,i} + 
			\delta_{\ell+1,i}^2
			).
		\end{equation}
		By \eqref{eq:Ineq_AGM} with $(a,b) = (\delta_{\ell,i}^2,
		\delta_{\ell+1,i}^2)$, we obtain
		\begin{equation}
			\label{eq:max_td}
			\max_{-\widebar \delta_i \leq \delta_{\ell,i},\delta_{\ell+1,i} \leq \widebar \delta_i}(
			\delta_{\ell,i}^2 + 
			\delta_{\ell,i}\delta_{\ell+1,i} + 
			\delta_{\ell+1,i}^2 ) = 3 \widebar \delta_i^2.
		\end{equation}
		From \eqref{eq:int_hxi_txi} and \eqref{eq:max_td},
		it follows that
		\[
		\int_{t_{\ell}}^{t_{\ell+1}} 
		|\widehat x_i(t) - \widebar x_i(t)|^2dt \leq 
		(t_{\ell+1} - t_{\ell} ) \widebar{\delta}_i^2.
		\]
		Combining this with \eqref{eq:integral_decomp},
		we conclude that 
		the assertion \eqref{eq:hxi_txi_error} holds.
	\end{proof}
	
	As the sampling interval decreases,
	the impact of the noise $(\delta_{\ell})_{\ell=0}^N$ 
	on the reconstructed derivative $\widebar x'$
	increases.
	Despite this sensitivity,
	the Gramian
	associated with
	the error derivative $x' - \widebar x'$ has
	an upper bound involving  $h^2$, $\|\widebar \delta\|_{\mathbb{R}^n}^2$, and $h\|\widebar \delta\|_{\mathbb{R}^n}$.
	The following theorem 
	provides this bound,
	incorporating
	the process noise $v$ in the context of data-driven control.
	\begin{theorem}
		\label{thm:noisy_sampled_data}
		Let $h>0$ be as in \eqref{eq:h_def}, and suppose that $x \in \Hso([0,\tau];\mathbb{R}^n)$ satisfies
		Assumption~\ref{assump:deri_bound}.
		Define $\widebar x \in \Hso([0,\tau];\mathbb{R}^n)$
		by
		\eqref{eq:tilde_x_def}. 
		Let $\sigma >0$ and $v \in \Ls([0,\tau];\mathbb{R}^n)$
		satisfy $VV^* \preceq \sigma^2 I_n$, where $V$ is the synthesis
		operator associated with $v$.
		Then 
		the synthesis operators $\widebar{P}$ and 
		$\widebar{Q}$
		associated with 
		$x' - \widebar x' + v$ and $x - \widebar x$,
		respectively, satisfy
		\begin{align}
			\widebar P \widebar P^* &\preceq
			\left(
			\frac{ h(M
				\|\widebar x\|_{\Ls}  +  c)}{\pi - hM}
			+
			\frac{\pi  \sqrt{\tau} \|\widebar{\delta}\|_{\mathbb{R}^n} }{\pi - hM} + \sigma
			\right)^2 I_n,
			\label{eq:tilP_ad_bound_coro}\\
			\widebar{Q}\widebar{Q}^* &\preceq
			\frac{\tau^2}{\pi^2}
			\left(
			\frac{ h(M
				\|\widebar x\|_{\Ls}  +  c)}{\pi - hM}
			+
			\frac{\pi  \sqrt{\tau} \|\widebar{\delta}\|_{\mathbb{R}^n}}{\pi - hM} 
			\right)^2 I_n.
			\label{eq:tilQ_ad_bound_coro}
		\end{align}
	\end{theorem}
	\begin{proof}
		We prove only 
		the inequality \eqref{eq:tilP_ad_bound_coro}
		for $\widebar P \widebar P^*$.
		The same argument can be applied to show
		the inequality \eqref{eq:tilQ_ad_bound_coro} for
		$\widebar Q \widebar Q^*$.

		Define 
		$\widehat x \in \Hso([0,\tau];\mathbb{R}^n)$
		by
		\eqref{eq:hat_x_def}. 
		Let $P$ and $\widebar P_0$  be the synthesis operators  
		associated with 
		$x' -\widehat x'$ and $\widehat x' - \widebar x'$, respectively.
		Using Lemma~\ref{lem:til_hat_x_diff}, 
		we obtain
		\begin{equation*}
			\frac{
				(M
				\|\widehat x\|_{\Ls} +  c)^2}{(\pi - hM)^2} \leq 
			\frac{
				(M
				\|\widebar x\|_{\Ls} +M\sqrt{\tau}\|\widebar{\delta}\|_{\mathbb{R}^n} +  c)^2}{(\pi - hM)^2}.
		\end{equation*}
		This together with Proposition~\ref{prop:PL_app} yields
		\begin{align}
			PP^* &\preceq h^2
			\left( \frac{
				M
				\|\widebar x\|_{\Ls} +M\sqrt{\tau}\|\widebar{\delta}\|_{\mathbb{R}^n}+  c}{\pi - hM}\right)^2 I_n.
			\label{eq:P_ad_bound_coro}
		\end{align}
		From Lemmas~\ref{lem:synthesis_ad_bound} and
		\ref{lem:til_hat_x_diff}, it also follows that
		\begin{equation}
			\label{eq:tilP0_ad_bound}
			\widebar P_0 \widebar P_0^* \preceq \tau
			\|\widebar{\delta}\|_{\mathbb{R}^n}^2 I_n.
		\end{equation}

		A similar argument to Remark~\ref{rem:young} shows that
		for all $\varepsilon_{p,1} > 0$,
		\[
		P \widebar P_0^* + \widebar P_0P^* \preceq \varepsilon_{p,1} PP^* + 
		\frac{\widebar P_0\widebar P_0^*}{\varepsilon_{p,1}},
		\]
		and analogous bounds can be obtained for
		$PV^*+VP^*$ and $\widebar P_0 V^* +V\widebar P_0^*$. Therefore,	
		$\widebar P=P+\widebar P_0+V$ satisfies
		\begin{equation}
			\label{eq:tilP_ad_bound}
			\widebar P \widebar P^* \preceq
			(1+\varepsilon_{p,1} + \varepsilon_{p,2}) PP^* + 
			\left(1+ \frac{1}{\varepsilon_{p,1}} + \varepsilon_{p,3}\right)   \widebar P_0 \widebar P_0^* + \left(1+ \frac{1}{\varepsilon_{p,2}} + \frac{1}{\varepsilon_{p,3} }\right) VV^*
		\end{equation}
		for all $\varepsilon_{p,1},\varepsilon_{p,2},\varepsilon_{p,3} > 0$.
		Combining this with \eqref{eq:P_ad_bound_coro}
		and \eqref{eq:tilP0_ad_bound},
		we obtain
		\begin{equation}
			\label{eq:tPtP*_bound}
			\widebar P \widebar P^* \preceq
			\left( 
			(1+\varepsilon_{p,1} + \varepsilon_{p,2})\alpha
			+
			\left(1+ \frac{1}{\varepsilon_{p,1}} + \varepsilon_{p,3}\right)
			\beta + 
			\left(1+ \frac{1}{\varepsilon_{p,2}} + \frac{1}{\varepsilon_{p,3} }\right) \gamma \right)I_{n},
		\end{equation}
		where 
		\begin{align*}
			\alpha &\coloneqq h^2
			\left( \frac{
				M
				\|\widebar x\|_{\Ls} +M\sqrt{\tau}\|\widebar{\delta}\|_{\mathbb{R}^n}+  c}{\pi - hM}\right)^2, \\
			\beta &\coloneqq \tau
			\|\widebar{\delta}\|_{\mathbb{R}^n}^2, \\
			\gamma &\coloneqq \sigma^2.
		\end{align*}
		
		Define 
		$f \colon (0,\infty)^3 \to \mathbb{R}$ by
		\begin{align*}
			f(\varepsilon_1,\varepsilon_2,\varepsilon_3) &
			\coloneqq 
			(1+\varepsilon_1 + \varepsilon_2) \alpha + 
			\left(1+ \frac{1}{\varepsilon_1} + \varepsilon_3\right)   \beta + \left(1+ \frac{1}{\varepsilon_2} + \frac{1}{\varepsilon_3 }\right) \gamma.
		\end{align*}
		Applying
		\eqref{eq:Ineq_AGM} to the three pairs
		\[
		(a,b) = \left(\varepsilon_1 \alpha,
		\frac{\beta}{
			\varepsilon_1
		}
		\right),\,
		\left(\varepsilon_2 \alpha,
		\frac{\gamma}{
			\varepsilon_2
		}
		\right),\,
		\left(\varepsilon_3 \beta,
		\frac{\gamma}{
			\varepsilon_3
		}
		\right),
		\]
		we obtain
		\[
		\inf_{\varepsilon_1,\varepsilon_2,\varepsilon_3>0}
		f(\varepsilon_1,\varepsilon_2,\varepsilon_3) =
		\left(\sqrt{\alpha}+\sqrt{\beta}+\sqrt{\gamma}\right)^2.
		\]
		From this and \eqref{eq:tPtP*_bound},
		the desired inequality \eqref{eq:tilP_ad_bound_coro} follows.
	\end{proof}
	
	\section{Example: $H_{\infty}$-control from noisy sampled data}
	\label{sec:example}
	The purpose of this section is to illustrate
	the proposed data-driven control method
	through a numerical example.
	From noisy sampled data,
	we design a state-feedback gain for 
	$H_{\infty}$-control of a linearized aircraft model~\cite{Alwi2008}
	in the continuous-time setting.
	In Section~\ref{sec:example_system}, we describe 
	the aircraft model and the performance criterion
	for $H_{\infty}$-control. 
	In Section~\ref{sec:data_generation}, 
	we explain how the noisy sampled data are generated.
	In Section~\ref{sec:example_controller_design}, 
	we apply 
	the informativity condition
	to the reconstructed state and the perturbed input, using
	the error bound for the reconstructed state.
	We then construct
	a feedback gain from the solution of the LMIs.
	
	For $\lambda_1,\dots,\lambda_n \in \mathbb{R}$, 
	we denote by $\diag \,(\lambda_1,\ldots,\lambda_n)$ 
	the diagonal matrix with diagonal entries
	$\lambda_1,\dots,\lambda_n$. For matrices $\Lambda_1,\dots,\Lambda_n$, we denote by 
	$\blkdiag\,(\Lambda_1,\ldots,\Lambda_n)$ the block
	diagonal matrix with $\Lambda_1,\dots,\Lambda_n$ 
	on its diagonal blocks and zeros elsewhere.
	
	\subsection{System}
	\label{sec:example_system}
	We consider a linearized aircraft model around an operating condition
	corresponding to a true airspeed of $184~\mathrm{m/s}$. The system matrices $A_0$
	and $B_0$ are given by
	\begin{align*}
		A_0 &=
		\begin{bmatrix}
			-0.6803 & 0.0002 & -1.0490 & 0 \\
			-0.1463 & -0.0062 & -4.6726 & -9.7942 \\
			1.0050 & -0.0006 & -0.5717 & 0 \\
			1 & 0 & 0 & 0
		\end{bmatrix},\\
		B_0 &=
		\begin{bmatrix}
			-1.5539 & 0.0154 \\
			0 & 1.3287 \\
			-0.0398 & -0.0007 \\
			0 & 0
		\end{bmatrix}.
	\end{align*}
	The state vector is $\xi = \begin{bmatrix}
		\xi_1 & \xi_2 & \xi_3 & \xi_4
	\end{bmatrix}^{\top}$, where
	$\xi_1$ is the pitch rate in radians per second, $\xi_2$ is 
	the true
	airspeed in meters per second, 
	$\xi_3$ is the angle of attack in radians, and $\xi_4$ is the
	pitch angle in radians. 
	The inputs are the elevator deflection in radians
	and the total thrust in Newtons, where the thrust input is scaled by $10^5$.
	Since the true airspeed $\xi_2$ is 
	larger in magnitude than the 
	other state components, 
	we introduce the scaled state $x = S\xi$,
	where
	$
	S=\diag\,(1,1/40,1,1).
	$
	Therefore, the system $(A,B)$ used for data generation and controller
	design is given by
	\[
	A=SA_0S^{-1}
	\quad 
	\text{and} \quad 
	B=SB_0.
	\]
	
	In the $H_\infty$-control problem,
	the weighting matrix $\Omega$ associated with 
	the
	external signals is chosen as
	\[
	\Omega =
	0.01\diag\,(1,\,1,\,5,\,1,\,0,\,0,\,0,\,0,\,1,\,0.1),
	\]
	which implies that
	state biases are not considered in this example.
	The matrices 
	$C$, $D$, and $E$
	for the performance output are given by
	\[
	C=
	\begin{bmatrix}
		\diag\,(1,\,1,\,5,\,3)\\
		0_{2\times 4}
	\end{bmatrix},
	\quad
	D=
	\begin{bmatrix}
		0_{4\times 2}\\
		\diag\,(0.1,\,0.01)
	\end{bmatrix},\quad 
	\text{and} \quad 
	E = 0_{6 \times 10}.
	\]
	
	\subsection{Data generation}
	\label{sec:data_generation}
	We generate $30$ independent state-input trajectories
	$(x_k,u_k)_{k=1}^{30}$ over the time interval $[0,0.4]$.
	For each $k=1,\dots,30$, the initial state $x_k(0)$ is randomly generated as
	\[
	x_k(0)=\diag\,(0.05,\,0.075,\,0.06,\,0.08)\eta,
	\]
	where each element of $\eta \in \mathbb{R}^4$ is  independently sampled
	from the uniform distribution on $(-1,1)$.
	The input $u_k$ is generated as 
	\[
	u_k(t) = 
	\sum_{\ell=1}^{10}
	\begin{bmatrix} 
		a_{k,\ell,1} \sin (2\pi f_{k,\ell,1} t + \varphi_{k,\ell,1} ) \\
		a_{k,\ell,2} \sin (2\pi f_{k,\ell,2} t + \varphi_{k,\ell,2} ) 
	\end{bmatrix},\quad 0 \leq t \leq 0.4.
	\]
	Here, the amplitudes $a_{k,\ell,i}$, the
	frequencies $f_{k,\ell,i}$, and the initial phases 
	$\varphi_{k,\ell,i}$
	are independently sampled from the uniform distributions on the intervals $(0,0.15/\sqrt{10})$, $(0,5)$, and 
	$(0,2\pi)$, respectively.

	For each $k=1,\dots,30$,
	the state $x_k$ satisfies the following differential equation in
	the presence of the 
	process noise $v_k$ and the input disturbance $r_k$:
	\[
	x_k'(t) = A x_k(t) + B (u_k(t)+r_k(t)) + v_k(t),\quad 0 \leq t \leq 0.4.
	\]
	The process noise $v_k$ is modeled as continuous-time white noise with power
	spectral density $10^{-7}$. The input disturbance $r_k$ is generated by
	passing white noise with the same power spectral density through a
	first-order low-pass filter with transfer function 
	$F(s) = 1/(s + 1)$. A warm-up
	interval of length $10$ is used before data collection to
	remove the transient effects of the filter.
	
	We consider periodic sampling with period
	$
	h =3.125\times 10^{-4}.
	$
	At each sampling time $t=\ell h$, $\ell=0,\dots,64$,
	we measure 
	the noisy state $x_k(\ell h) + \delta_{k,\ell}$, where
	$\delta_{k,\ell}$ is the measurement noise.
	The componentwise 
	upper bound $\widebar \delta$ on the measurement noise is given by
	\[
	\widebar \delta =
	\begin{bmatrix}
		\widebar \delta_1 \\
		\widebar \delta_2 \\
		\widebar \delta_3 \\
		\widebar \delta_4
	\end{bmatrix}
	= 10^{-4}
	\begin{bmatrix}
		0.5 \\
		1.0 \\
		0.6 \\
		0.8
	\end{bmatrix}.
	\]
	For each $k=1,\dots,30$, 
	$\ell = 0,\dots,64$, and $i=1,\ldots,4$, the $i$-th element of
	the measurement error $\delta_{k,\ell}$ is generated
	independently from the uniform distribution on the interval $(-\widebar{\delta}_i, \widebar{\delta}_i)$ at each sampling instant.

	\subsection{Controller design}
	\label{sec:example_controller_design}
	Let $\widebar x_k$ be the state  reconstructed 
	from the noisy sampled data 
	$(x_k(\ell h)+\delta_{k,\ell})_{\ell=0}^{64}$ as in 
	\eqref{eq:tilde_x_def}.
	Let $V_k$ and $R_k$ be the synthesis operators associated with
	the process noise $v_k$ and the input disturbance $r_k$, respectively.
	Setting the constants
	$
	\sigma_v = 1.5\times 10^{-4}$ and 
	$\sigma_r = 5\times 10^{-5}$, we numerically
	verify that the noise and disturbance signals 
	$(v_k)_{k=1}^{30}$ and 
	$(r_k)_{k=1}^{30}$ satisfy
	\[
	\sum_{k=1}^{30}V_kV_k^* \preceq \sigma_v I_4\quad 
	\text{and} \quad 
	\sum_{k=1}^{30} R_kR_k^* \preceq \sigma_r I_2.
	\]
	Furthermore, the constants $M$ and $c$ used to bound the state derivative in Assumption~\ref{assump:deri_bound} are chosen as
	$
	M=1.6
	$
	and 
	$
	c=0.12.
	$
	
	By applying Theorem~\ref{thm:noisy_sampled_data}
	and Remark~\ref{rem:young} with
	\[
	(\varepsilon_1,\varepsilon_2,\varepsilon_3) = 
	(8.401 \times 10^{-3}, 8.028\times 10^{-4}, 9.555\times 10^{-3}),
	\]
	we set the noise-intensity matrix $\Theta$ to
	\[
	\Theta = 10^{-4}
	\blkdiag\,
	(5.872\times 10^{-2}I_4,\, 4.930 \times 10^{-2}I_4,\,
	1.013  I_2
	),
	\]
	which satisfies
	Assumption~\ref{assump:Theta} for
	the data $\mathfrak{D} = (\widebar x_k,u_k)_{k=1}^{30}$
	and the noise 
	\[
	(\widebar x_k'- x_k'  + v_k,\widebar x_k - x_k,-r_k).
	\]
	The LMIs in Theorem~\ref{thm:H_inf} are solved using
	MATLAB R2026a with YALMIP~\cite{Lofberg2004} and 
	MOSEK~\cite{MOSEK2026}. Then we find that
	the data $\mathfrak{D}$
	are informative for $H_{\infty}$-control with
	performance $\gamma =1.35$  under the noise class $\Delta_{\tau,\Theta}$.
	From the solutions $\Phi$ and $L$ of the LMIs,  the feedback gain $K$ is computed as
	\[
	K = L\Phi^{-1} = 
	\begin{bmatrix}
		5.737  &-24.46  &  2.493 &   8.415 \\
		0.7779 &  -4.585 &   0.8685  &  1.090
	\end{bmatrix}.
	\]
	In the no-control case $K = 0_{2 \times 4}$,
	we have $\|G_K\|_{H_\infty} = 68.99$.
	Therefore, the feedback gain obtained from
	the noisy sampled data improves the $H_{\infty}$-control performance.
	
	\section{Conclusion}
	\label{sec:conclusion}
	In this paper, we studied
	derivative-free data-driven control for continuous-time systems. We considered a broad noise class that covers process noise, measurement noise, and input disturbances. By embedding the state-input data into synthesis operators, we characterized the set of all data-consistent systems in terms of a QMI.
	Based on this characterization, we obtained a necessary 
	and sufficient LMI condition under which noisy data are informative for quadratic stabilization, and
	extended this result to $H_2$-control and $H_\infty$-control.
	From the viewpoint of synthesis operators,
	we also analyzed the error that arises when continuous-time signals are reconstructed from sampled data. Combining this error analysis with the proposed informativity framework,
	we can design controllers directly from noisy sampled data for continuous-time systems. 
	Future work includes extending the proposed method to the case where only input-output data are available.

	\appendix
	\section{Norm of a sum of synthesis operators}
	Let $\ell^2(\mathbb{N};\mathbb{R}^k)$
	denote the space of sequences 
	$\zeta=(\zeta_j)_{j\in\mathbb{N}}$ with $\zeta_j\in\mathbb{R}^k$ satisfying
	$\sum_{j=1}^{\infty} \|\zeta_j\|_{\mathbb{R}^k}^2 < \infty$.
	This space is equipped with the inner product
	\[
	\langle \zeta, \xi \rangle_{\ell^2} =
	\sum_{j=1}^{\infty}  \langle \zeta_j,\xi_j \rangle_{\mathbb{R}^k}
	\]
	for $\zeta=(\zeta_j)_{j\in\mathbb{N}}$ and 
	$\xi=(\xi_j)_{j\in\mathbb{N}}$ in 
	$\ell^2(\mathbb{N};\mathbb{R}^k)$.
	For brevity, we write $\ell^2(\mathbb{N})
	\coloneqq \ell^2(\mathbb{N};\mathbb{R})$.
	
	For $k=1,\dots,N$, 
	let $F_k \in  \mathcal{L}(\Hs  [0,\tau_k], \mathbb{R}^n)$ be the synthesis operator associated with $f_k \in \Ls ([0,\tau_k]; \mathbb{R}^n)$.
	Set $\tau \coloneqq (\tau_k)_{k=1}^{N}$, and define
	$F \in \mathcal{L}(\Ht, \mathbb{R}^n)$ by
	\[
	F\phi \coloneqq \sum_{k=1}^{N} F_k\phi_k,
	\quad 
	\phi = (\phi_k)_{k=1}^{N} \in \Ht.
	\]
	For $k=1,\dots, N$, 
	define the orthonormal basis $(\psi_{k,j})_{j \in \mathbb{N}}$ of $\Ls[0,\tau_k]$ by
	\[
	\psi_{k,j}(t) \coloneqq \sqrt{\frac{2}{\tau_k}} \sin\left(
	\frac{j \pi }{\tau_k} t
	\right),\quad t \in [0,\tau_k],~j \in \mathbb{N}.
	\]
	For $j \in \mathbb{N}$, 
	define $M_j \in \mathbb{R}^{N\times n}$ by
	\[
	M_j \coloneqq \left[
	\frac{\tau_k \langle f_{k,\ell}, \psi_{k,j} \rangle_{\Ls}}{j \pi}
	\right]_{1 \leq k \leq N,\, 1 \leq \ell \leq n},
	\]
	where $f_{k,\ell}$ is the $\ell$-th element of 
	$f_k $.

	We now  provide
	a representation of $\|F\|$ in terms of
	the matrix sequence $(M_j)_{j\in \mathbb{N}}$.
	Note that in the definition of $M_j$, 
	the Fourier coefficient $\langle f_{k,\ell}, \psi_{k,j} \rangle_{\Ls}$ is 
	scaled by a factor proportional to $1/j$. 
	Since the index $j$ corresponds to the 
	frequency of the sine function, 
	this indicates that the influence of 
	high-frequency noise components
	is smaller than that of low-frequency ones.
	From the next proposition, we see that 
	$\|F\|$ is small if
	the energy of 
	$(f_k)_{k=1}^{N}$
	is concentrated in a high-frequency range.
	This extends 
	the single-trajectory case \cite[Proposition~A]{Wakaiki2025Cont} to
	the multiple-trajectory case.

	\begin{proposition}
		Define $F \in \mathcal{L}(\Ht, \mathbb{R}^n)$ 
		and $M_j \in \mathbb{R}^{N\times n}$
		as above.
		Then
		\begin{equation*}
			\|F\| = \lim_{m\to \infty} \left\|\begin{bmatrix}
				M_{1} \\ \vdots \\ M_{m}
			\end{bmatrix}\right\|.
		\end{equation*}
	\end{proposition}
	\begin{proof}
		Since $\|F\| = \|F^*\|$,
		it suffices to prove that
		\begin{equation}
			\label{eq:synthesis_ad_norm}
			\|F^*\| = \lim_{m\to \infty} \left\|\begin{bmatrix}
				M_{1} \\ \vdots \\ M_{m}
			\end{bmatrix}\right\|.
		\end{equation}
		Since
		\[
		FF^* = \sum_{k=1}^{N} F_kF_k^*,
		\]
		it follows that for all $\eta \in \mathbb{R}^n$,
		\begin{equation}
			\label{eq:T_sum_norm}
			\|F^* \eta\|_{\Ht}^2 = \langle F^*\eta, F^*\eta \rangle_{\Ht} =
			\langle 
			\eta, FF^* \eta
			\rangle_{\mathbb{R}^n} =
			\eta^{\top} FF^* \eta = \sum_{k=1}^{N} \|F_k^* \eta\|_{\Hs}^2.
		\end{equation}
		For $k=1,\dots, N$, 
		define the bounded linear operator $\Theta_k \colon \mathbb{R}^{n} \to \ell^2(\mathbb{N})$ by
		\[
		\Theta_{k} \eta \coloneqq 
		\left(
		\sum_{\ell=1}^{n}
		\frac{
			\tau_k
			\langle f_{k,\ell}, \psi_{k,j} \rangle_{\Ls}}{j \pi }\eta_\ell
		\right)_{j \in \mathbb{N}},\quad 
		\eta = \begin{bmatrix}
			\eta_1 \\ \vdots \\ \eta_{n}
		\end{bmatrix} \in \mathbb{R}^{n}.
		\]
		Then, as seen in the proof of the single-trajectory case
		\cite[Proposition~A]{Wakaiki2025Cont},
		we have
		\begin{equation}
			\label{eq:Tk*_norm}
			\|F_k^* \eta\|_{\Hs}=
			\|\Theta_k \eta \|_{\ell^2}
		\end{equation}
		for all $\eta \in \mathbb{R}^n$.
		Define the bounded linear operator $\Theta \colon \mathbb{R}^n\to \ell^2(\mathbb{N};\mathbb{R}^{N})$ by
		\[
		\Theta \eta  \coloneqq 
		\left(
		\begin{bmatrix}
			(\Theta_1 \eta)_j \\ 
			\vdots  \\
			(\Theta_{N} \eta)_j
		\end{bmatrix}
		\right)_{j \in \mathbb{N}},\quad 
		\eta \in \mathbb{R}^n,
		\]
		where $(\Theta_{k} \eta)_j$ denotes
		the $j$-th element of $\Theta_{k} \eta \in
		\ell^2(\mathbb{N})$ for $k=1,\dots,N$.
		Then
		\begin{equation}
			\label{eq:Lambda_norm}
			\sum_{k=1}^{N} 
			\|\Theta_k\eta \|_{\ell^2}^2 =
			\sum_{k=1}^{N} \sum_{j=1}^{\infty} |(\Theta_k \eta)_j|^2
			= 
			\sum_{j=1}^{\infty}  \left\|
			\begin{bmatrix}
				(\Theta_1 \eta)_j \\ 
				\vdots  \\
				(\Theta_{N} \eta)_j
			\end{bmatrix}
			\right\|_{\mathbb{R}^{N}}^2 =
			\|\Theta \eta \|_{\ell^2}^2
		\end{equation}
		for all $\eta \in \mathbb{R}^n$.
		Combining \eqref{eq:T_sum_norm}--\eqref{eq:Lambda_norm},
		we obtain
		\begin{equation}
			\label{eq:T*Lambda}
			\|F^*\| = 	\|\Theta \|.
		\end{equation}
		
		For $m \in \mathbb{N}$, let 
		$\Pi_m$ be the truncation operator on 
		$\ell^2(\mathbb{N};\mathbb{R}^{N})$ 
		such that for $\zeta = (\zeta_j)_{j \in \mathbb{N}}\in  \ell^2(\mathbb{N};\mathbb{R}^{N})$ and
		$j \in \mathbb{N}$,
		the $j$-th element $(\Pi_m \zeta )_j$ is defined by
		\[
		(\Pi_m \zeta )_j \coloneqq \begin{cases}
			\zeta_j, & j \leq m, \\
			0, & j >m.
		\end{cases}
		\]
		Let $\varepsilon >0$ be arbitrary.
		For each $i=1,\dots,n$, there exists $m_i \in \mathbb{N}$
		such that for all $m \geq m_i$,
		\[
		\|(\Theta - \Pi_m\Theta) e_i\|_{\ell^2} \leq \varepsilon,
		\]
		where $e_i$ is the $i$-th canonical basis vector of $\mathbb{R}^n$.
		Since this implies that for all $m \geq \max\{m_1,\dots,m_n \}$,
		\[
		\|\Theta - \Pi_m\Theta \| \leq \varepsilon  \sqrt{n},
		\]
		we have
		\begin{equation}
			\label{eq:lambda_PiN}
			\lim_{m\to \infty} \|\Theta - \Pi_m\Theta\| = 0.
		\end{equation}
		
		By construction,
		\[
		M_j \eta = 
		\begin{bmatrix}
			(\Theta_1 \eta)_j \\ 
			\vdots  \\
			(\Theta_{N} \eta)_j
		\end{bmatrix}
		\]
		for all $j \in \mathbb{N}$ and $\eta \in \mathbb{R}^n$.
		This yields
		\[
		\|\Pi_m \Theta \eta\|_{\ell^2}^2 = 
		\sum_{j=1}^m \|M_j \eta\|_{\mathbb{R}^{N}}^2 =
		\left\|\begin{bmatrix}
			M_{1} \\ \vdots \\ M_{m}
		\end{bmatrix} \eta \right\|_{\mathbb{R}^{mN}}^2
		\]
		for all $m \in \mathbb{N}$ and $\eta \in \mathbb{R}^n$.
		Therefore,
		\begin{equation}
			\label{eq:PiN_M}
			\|\Pi_m \Theta\| = 
			\left\|\begin{bmatrix}
				M_{1} \\ \vdots \\ M_{m}
			\end{bmatrix}  \right\|
		\end{equation}
		for all $m \in \mathbb{N}$.
		From \eqref{eq:T*Lambda}--\eqref{eq:PiN_M},
		we conclude that \eqref{eq:synthesis_ad_norm} holds.
	\end{proof}
	\printbibliography
	\end{document}